\documentclass[10pt]{article}
\usepackage{delarray} 
\usepackage{amscd}
\usepackage{amssymb}
\usepackage{amsmath}
\usepackage{amsthm,amsxtra}
\usepackage{latexsym,amsmath,mathrsfs}
\usepackage[bookmarks,colorlinks]{hyperref}
\usepackage[all]{xy}
\usepackage{amsfonts}
\usepackage{color}
\usepackage{fancybox}
\usepackage{colordvi}
\usepackage{multicol}
\usepackage{textcomp}
\usepackage{stmaryrd}
\usepackage[dvips]{graphicx}
\usepackage{enumerate,xspace}
\usepackage{geometry}
\usepackage{tocbibind}

\newcommand{\cris}{\mathrm{cris}}
\newcommand{\Fil}{{\mathrm{Fil}}}
\newcommand{\Ext}{\mathrm{Ext}}
\newcommand{\PGam}{{\Phi \Gamma}}
\newcommand\Hom{\mathrm{Hom}}
\def\st{\mathrm{st}}
\def\dR{\mathrm{dR}}
\def\rig{\mathrm{rig}}
\def\cyc{\mathrm{cyc}}

\def\gr{\mathrm{gr}}
\def\log{\mathrm{log}}

\def\GL{\mathrm{GL}}

\def\Max{\mathrm{Max}}
\def\max{\mathrm{max}}

\def\BX{\mathbf{X}}
\def\BB{\mathbf{B}}
\def\BZ{\mathbf{Z}}

\def\BD{\mathbf{D}}
\def\BQ{\mathbf{Q}}
\def\BN{\mathbf{N}}
\def\BF{\mathbf{F}}
\def\BA{\mathbf{A}}

\def\BC{\mathbf{C}}
\def\BR{\mathbf{R}}

\def\RE{\mathbf{E}}
\def\RA{\mathbf{A}}
\def\RW{\mathbf{W}}
\def\RB{\mathbf{B}}

\def\CL{\mathcal{L}}

\def\CD{\mathcal{D}}
\def\CV{\mathcal{V}}

\def\CF{\mathcal{F}}

\def\fo{\mathfrak{o}}

    \theoremstyle{plain}

    \newtheorem{thm}{Theorem}[section]
    
    \newtheorem{lem}[thm]{Lemma}
    
    \newtheorem{prop}[thm]{Proposition}

    \theoremstyle{definition}

    \newtheorem{defn}[thm]{Definition}

    \theoremstyle{remark}
    \newtheorem {rem}[thm]{Remark}

    \numberwithin{equation}{section}

\begin{document}

\title{A generalization of Colmez-Greenberg-Stevens formula}
\author{Bingyong Xie \footnote{This paper is supported by
Science and Technology Commission of Shanghai Municipality (grant
no. 13dz2260400) and the National Natural Science Foundation of
China (grant no. 11671137).}
\\ \small Department of Mathematics, East China Normal University,
Shanghai, China \\ \small byxie@math.ecnu.edu.cn}

\date{}
\maketitle

\begin{abstract} In this paper we study the derivatives of Frobenius
and the derivatives of Hodge-Tate weights for families of Galois
representations with triangulations. We give a generalization of the
Fontaine-Mazur $\CL$-invariant and use it to build a formula which
is a generalization of the Colmez-Greenberg-Stevens formula.
\end{abstract}




\section{Introduction}

In their remarkable paper \cite{MTT}, Mazur, Tate and Teitelbaum
proposed a conjectural formula for the derivative at $s=1$ of the
$p$-adic $L$-function of an elliptic curve $E$ over $\BQ$ when $p$
is a prime of split multiplicative reduction. An important quantity
in this formula is the so called $\CL$-invariant, namely
$\CL(E)=\log_p(q_E)/v_p(q_E)$ where $q_E\in  \BQ_p^\times$ is the
Tate period for $E$. This conjectural formula was proved by
Greenberg and Stevens \cite{GS} using Hida's families. Indeed, for
the weight $2$ newform $f$ attached to $E$, there exists a family of
$p$-adic ordinary Hecke eigenforms containing $f$. A key formula
they proved is \begin{equation} \CL(E)= -2
\frac{\alpha'(f)}{\alpha(f)} \label{eq:L-elliptic}
\end{equation} where $\alpha$ is the function of $U_p$-eigenvalues of
the eigenforms in the Hida family. On the other hand, they showed
that $-2\frac{\alpha'(f)}{\alpha(f)}$ is equal to $\frac{L'_p(f,
1)}{L(f,1)}$. Combining these two facts they obtained the
conjectural formula.

In this paper we will focus on (\ref{eq:L-elliptic}) which was later
generalized by Colmez \cite{Cz2010} to the non-ordinary setting. We
state Colmez's result below.

\begin{thm}\label{thm:colmez} $($\cite{Cz2010}$)$ Suppose that, at each closed point $z$ of $\mathrm{Max}(S)$
one of the Hodge-Tate weight of $\CV_z$ is $0$, and there exists
$\alpha\in S$ such that
$(\BB_{\cris,S}^{\varphi=\alpha}\widehat{\otimes}_S \CV
)^{G_{\BQ_p}}$ is locally free of rank $1$ over $S$. Suppose $z_0$
is a closed point of $\mathrm{Max}(S)$ such that $\CV_{z_0}$ is
semistable with Hodge-Tate weights \footnote{In this paper, the
Hodge-Tate weights are defined to be minus the generalized
eigenvalues of Sen's operators. In particular the Hodge-Tate weight
of the cyclotomic character $\chi_\cyc$ is $-1$.} $0$ and $k\geq 1$.
Then the differential
$$\frac{\mathrm{d}\alpha}{\alpha} - \frac{1}{2} \CL \mathrm{d}\kappa+
\frac{1}{2}\mathrm{d}\delta $$ is zero at $z_0$, where $\CL$ is the
Fontaine-Mazur $\CL$-invariant of $\CV_{z_0}$.
\end{thm}

See \cite{Cz2010} for the precise meanings of $\kappa$ and $\delta$.
Roughly speaking, $\mathrm{d}\delta$ is the derivative of Frobenius,
and $\mathrm{d}\kappa$ is the derivative of Hodge-Tate weights.

The condition that
``$(\BB_{\cris,S}^{\varphi=\alpha}\widehat{\otimes}_S \CV
)^{G_{\BQ_p}}$ is locally free of rank $1$ over $S$'' in Theorem
\ref{thm:colmez} is equivalent to that $\CV$ admits a triangulation
\cite{Cz2008}.
 So, Theorem \ref{thm:colmez} means that the derivatives of
Frobenius and the derivatives of Hodge-Tate weights of a family of
$2$-dimensional representations of $G_{\BQ_p}$ with a triangulation
satisfy a non-trivial relation at each semistable (but
non-crystalline) point.

Colmez's theorem was generalized by Zhang \cite{Zhang} for families
of $2$-dimensional Galois representations of $G_K$ ($K$ a finite
extension of $\BQ_p$) and Pottharst \cite{Pott} who considered
families of (not necessarily \'etale) $(\varphi,\Gamma)$-modules of
rank $2$ instead of families of $2$-dimensional Galois
representations.

In this paper we give a generalization of Colmez's theorem which
includes the above generalizations as special cases.

Fix a finite extension $K$ of $\BQ_p$. What we work with is a family
of $K$-$B$-pair (called $S$-$B$-pair in our context) that is locally
triangulable. We will provide conditions for Fontaine-Mazur
$\CL$-invariant to be defined. Note that, the $\CL$-invariant is now
a vector with component number equal to $[K:\BQ_p]$.

\begin{thm}\label{thm:main}
Let $W$ be an $S$-$B$-pair that is semistable at a point $z\in
\Max(S)$. Suppose that $W$ is locally triangulable at $z$ with the
local triangulation parameters $(\delta_1,\cdots,\delta_n)$. Assume
that for $D_z$, the filtered $E$-$(\varphi,N)$-module attached to
$W_z$, the Fontaine-Mazur $\CL$-invariant $\vec{\mathcal{L}}_{s,t}$
$($see Definition \ref{defn:Fontaine-Mazur}$)$ can be defined for
 $ s, t \in \{1,2,\cdots, n\}$.  Then
$$ \frac{1}{[K:\BQ_p]}\left(\frac{\mathrm{d} \delta_t(p)}{\delta_t(p)} - \frac{\mathrm{d}
\delta_s(p)}{\delta_s(p)}\right)
 + \vec{\CL}_{s,t} \cdot (\mathrm{d}\vec{w}(\delta_t)-\mathrm{d}\vec{w}(\delta_s)) = 0 . $$
Here, $\vec{w}(\delta_i)$ is the Hodge-Tate weight of the character
$\delta_i$.
\end{thm}

In \cite{Xie2015} we proved Theorem \ref{thm:main} for a special
case, where we consider the case of $K=\BQ_p$ and demand that the
Frobenius is simisimple at $z$. The motivation and some potential
applications of our theorem was also discussed in \cite{Xie2015}.

Our paper is orginized as follows. In Section \ref{sec:B-pairs} we
recall the theory of $B$-pairs built by Berger. Then in Section
\ref{sec:family-B-pairs} we extend a part of this theory to families
of $B$-pairs, and discuss the relation between triangulations of
semistable $B$-pairs and refinements of their associated filtered
$(\varphi,N)$-modules. In Section \ref{sec:coh} we compare
cohomology groups of $(\varphi,\Gamma)$-modules and those of
$B$-pairs, and then attach a $1$-cocycle to each infinitesimal
deformation of a $B$-pair. In Section \ref{sec:aux-for} we use the
reciprocity law to build an auxiliary formula for $L$-invariants.
The $L$-invariant is defined in Section \ref{sec:L-inv}. In Section
\ref{sec:proj-van-prop} we prove a formula called ``projection
vanishing property'' for the above $1$-cocycle. Finally in Section
\ref{sec:main} we use the auxiliary formula in Section
\ref{sec:aux-for} and the projection vanishing property to deduce
Theorem \ref{thm:main}.

\section*{Notations}

Let $K$ be a finite extension of $\BQ_p$, $G_K$ the absolute Galois
group $\mathrm{Gal}(\overline{K}/K)$. Let $K_0$ be the maximal
absolutely unramified subfield of $K$. Let $G_K^{\mathrm{ab}}$
denote the maximal abelian quotient of $G_K$.

Let $\chi_\cyc$ be the cyclotomic character of $G_K$, $H_K$ the
kernel of $\chi_\cyc$ and $\Gamma_K$ the quotient $G_K/H_K$. Then
$\chi_\cyc$ induces an isomorphism from $\Gamma_K$ onto an open
subgroup of $\BZ_p^\times$.

Let $E$ be a finite extension of $K$ such that all embeddings of $K$
into an algebraic closure of $E$ are contained in $E$,
$\mathrm{Emb}(K,E)$ the set of embeddings of $K$ into $E$. We
consider $E$ as a coefficient field and let $G_K$ acts trivially on
$E$.

Let $\mathrm{rec}_K$ be the reciprocity map of local class field
theory such that $\mathrm{rec}_K(\pi_K)$ is a lifting of the inverse
of $q$th power Frobenius of $k$, where $\pi_K$ is a uniformizing
element of $K$ and $k$ is the residue field of $K$ with cardinal
number $q$. Note that the image of $\mathrm{rec}_K$ coincides with
the image of the Weil group $W_K\subset G_K$ by the quotient map
$G_K\rightarrow G_K^{\mathrm{ab}}$. Let $\mathrm{rec}_K^{-1}: W_K
\rightarrow K^\times$ be the converse map of $\mathrm{rec}_K$.

\section{$(\varphi,\Gamma_K)$-modules and $B$-pairs}
\label{sec:B-pairs}

\subsection{Fontaine's rings}

We recall the construction of Fontaine's period rings. Please
consult \cite{Fontaine, Ber02} for more details.

Let $\BC_p$ be a completed algebraic closure of $\BQ_p$ with
valuation subring $\mathfrak{o}_{\BC_p}$ and $p$-adic valuation
$v_p$ normalized such that $v_p(p)=1$.

Let $\widetilde{\RE}$ be $\{ (x^{(i)})_{i\geq 0} \ | \ x^{(i)}\in
\BC_p, \ (x^{(i+1)})^p=x^{(i)} \ \forall \  i\in\BN \},$ and let
$\widetilde{\RE}^+$ be the subset of $\widetilde{\RE}$ such that
$x^{(0)}\in \mathfrak{o}_{\BC_p}$. If $x,y\in \widetilde{\RE}$, we
define $x+y$ and $xy$ by
$$ (x+y)^{(i)}=\lim_{j\rightarrow \infty} (x^{(i+j)}+y^{(i+j)})^{p^j} , \hskip 10pt (xy)^{(i)}=x^{(i)}y^{(i)} .
$$ Then $\widetilde{\RE}$ is a field of characteristic $p$.  Define a function
$v_\RE:\widetilde{\RE}\rightarrow \: \BR\cup\{+\infty\}$ by putting
$v_\RE((x^{(n)}))=v_p(x^{(0)})$. This is a valuation for which
$\widetilde{\RE}$ is complete and $\widetilde{\RE}^+$ is the ring of
integers in $\widetilde{\RE}$.  If we let
$\varepsilon=(\varepsilon^{(n)})$ be an element of
$\widetilde{\RE}^+$ with $\epsilon^{(0)}=1$ and $\epsilon^{(1)}\neq
1$, then $\widetilde{\RE}$ is a completed algebraic closure of
$\BF_p((\varepsilon-1))$. Put $\omega= [\varepsilon]-1$. Let
$\tilde{p}$ be an element of $\widetilde{\RE}$ such that
$\tilde{p}^{(0)}=p$.

Let $\widetilde{\RA}^+$ be the ring $\RW(\widetilde{\RE}^+)$ of Witt
vectors with coefficients in $\widetilde{\RE}^+$, $\widetilde{\RA}$
the ring of Witt vectors $\RW(\widetilde{\RE})$, and
$\widetilde{\RB}^+=\widetilde{\RA}[1/p]$. The map $$\theta:
\widetilde{\RB}^+\rightarrow \BC_p, \hskip 10pt \sum_{n\gg -\infty}
p^k[x_k] \mapsto \sum_{n\gg -\infty} p^k x_k^{(0)} $$ is surjective.
Let $\RB^+_\dR$ be the $\ker(\theta)$-adic completion of
$\widetilde{\RB}^+$. Then $t_\cyc=\log[\varepsilon]$ is an element
of $\RB^+_\dR$, and put $\RB_\dR=\RB^+_\dR[1/t_\cyc]$. There is a
filtration $\Fil^\bullet$ on $\RB_\dR$ such that $\Fil^i
\RB_\dR=\bigoplus_{j\geq i} \RB^+_\dR t_\cyc^j$.

Let $\RB^+_\max$ be the subring of $\widetilde{\RB}^+$ consisting of
elements of the form $\sum_{n\geq 0} b_n ([\tilde{p}]/p)^n$, where
$b_n\in \widetilde{\RB}^+$ and $b_n\rightarrow 0$ when $n\rightarrow
+\infty$. Put $\RB_\max=\RB_\max^+[1/t_\cyc]$; $\RB_\max$ is
equipped with a $\varphi$-action. Put
$\RB_\log=\RB_\max[\log[\tilde{p}]]$;
 $\RB_\log$ is equipped with a
$\varphi$-action and a monodromy $N$; $\RB_\log^{N=0}=\RB_\max$;
$\RB_\log$ is a subring of $\RB_\dR$. Put
$\RB_e=\RB_\max^{\varphi=1}$. We have the following fundamental
exact sequence
\[ \xymatrix{  0 \ar[r] & \BQ_p \ar[r] & \RB_e \ar[r] & \RB_\dR/\RB^+_\dR \ar[r] & 0 .  }\]

If $r$ and $s$ are two elements in $\BN[1/p]\cup \{+\infty\}$, we
put $\widetilde{\RA}^{[r,s]} =\widetilde{\RA}^+\{
\frac{p}{[\bar{\omega}^r]}, \frac{[\bar{\omega}^s]}{p}  \} $ and
$\widetilde{\RB}^{[r,s]}=\widetilde{\RA}^{[r,s]}[1/p]$ with the
convention that $p/[\bar{\omega}^{+\infty}]=1/[\bar{\omega}]$ and
$[\bar{\omega}^{+\infty}]/p=0$. We equip these rings with the
$p$-adic topology. There are natural continuous $G_K$-actions on
$\widetilde{A}_{[r,s]}$ and $\widetilde{B}_{[r,s]}$. Frobenius
induces isomorphisms $\varphi:
\widetilde{A}_{[r,s]}\xrightarrow{\sim} \widetilde{A}_{[pr,ps]}$ and
$\varphi: \widetilde{B}_{[r,s]}\xrightarrow{\sim}
\widetilde{B}_{[pr,ps]}$. If $r\leq r_0\leq s_0\leq s$, then we have
the $G_K$-equivariant injective natural map $\widetilde{A}_{[r,
s]}\hookrightarrow \widetilde{A}_{[r_0,s_0]}$. For $r>0$ we put
$\widetilde{\BB}^{\dagger, r}_\rig= \bigcap_{s\in [r,+\infty)}
\widetilde{B}_{[r,s]}$ (equipped with certain Frechet topology) and
$\widetilde{\BB}^\dagger_\rig= \cup_{r>0}\widetilde{\BB}^{\dagger,
r}_\rig$ (equipped with the inductive limit topology). Frobenius
induces isomorphisms $\varphi: \widetilde{\BB}^{\dagger,
r}_\rig\xrightarrow{\sim} \widetilde{\BB}^{\dagger, pr}_\rig$ and
$\varphi: \widetilde{\BB}^{\dagger}_\rig\xrightarrow{\sim}
\widetilde{\BB}^{\dagger}_\rig$.

Put $$ A_{K'_0}=\{\sum_{k\geq -\infty}^{+\infty}a_k\omega^k \ | \
a_k\in \mathfrak{o}_{K'_0} , \ a_k\rightarrow 0 \ \text{ when
}k\rightarrow-\infty)\} $$ and $B_{K'_0}=A_{K'_0}[1/p]$. Here $K'_0$
is the maximal absolutely unramified subfield of
$K_\infty=K(\mu_{p^\infty})$. Then $A_{K'_0}$ is a complete discrete
valuation ring with $p$ as a prime element, and $B_{K'_0}$ is the
fractional field of $A_{K'_0}$. The $G_K$-action and $\varphi$
preserve $A_{K'_0}$: $\varphi(\omega)=(1+\omega)^p-1$ and
$g(\omega)=(1+\omega)^{\chi_\cyc(g)}-1$. Let $\BA$ be the $p$-adic
completion of the maximal unramified extension of $A_{K'_0}$ in
$\widetilde{\BA}$, $\BB$ its fractional field. Then $\varphi$ and
the $G_K$-action preserve $\BA$ and $\BB$.

We put $\BB_K=\BB^{H_K}$ and $\BB^{\dagger,r}_K=\BB_K\cap
\widetilde{\BB}^{\dagger, r}$. Let $\BB^{\dagger,r}_{\rig,K}$ be the
Frechet completion of $\BB^{\dagger,r}_K$ for the topology induced
from that on $\widetilde{\BB}^{\dagger,r}_\rig$, and put
$\BB^\dagger_{\rig,K}=\cup_{r>0} \BB^{\dagger,r}_{\rig,K}$ equipped
with the inductive limit topology. Frobunius induces injections
$\BB^{\dagger,r}_{\rig,K}\hookrightarrow \BB^{\dagger,pr}_{\rig,K}$
and $\BB^{\dagger}_{\rig,K}\hookrightarrow\BB^{\dagger}_{\rig,K}$;
there are continuous $\Gamma_K$-actions on
$\BB^{\dagger,r}_{\rig,K}$ and $\BB^{\dagger}_{\rig,K}$.

\vskip 5pt

We end this subsection by the definition of
$E$-$(\varphi,\Gamma_K)$-modules \cite{Naka}.

\begin{defn} An $E$-$(\varphi,\Gamma_K)$-module is a finite
$\BB^{\dagger}_{\rig,K}\otimes_{\BQ_p}E$-module $M$ equipped with a
Frobenius semilinear action $\varphi_M$ and a comtinuous semilinear
$\Gamma_K$-action such that $M$ is free as a
$\BB^{\dagger}_{\rig,K}$-module, that
$\mathrm{id}_{\BB^\dagger_{\rig,K}}\otimes \varphi_M :
\BB^\dagger_{\rig,K}\bigotimes_{\varphi, \BB^\dagger_{\rig,K}} M
\rightarrow M $ is an isomorphism, and that $\varphi_M$ and the
$\Gamma_K$-action commute with each other.
\end{defn}

By \cite[Lemma 1.30]{Naka} if $M$ is an
$E$-$(\varphi,\Gamma_K)$-module, 
then $M$ is free over $\BB^\dagger_{\rig,K}\otimes_{\BQ_p}E$.

\subsection{$B$-pairs}

We recall the theory of $E$-$B$-pairs \cite{Ber08, Naka}.

Put $\BB_{e,E}=\BB_e\otimes_{\BQ_p}E$,
$\BB^+_{\dR,E}=\BB^+_\dR\otimes_{\BQ_p}E$ and $\BB_{\dR,
E}=\BB_\dR\otimes_{\BQ_p}E$. We extend the $G_K$-actions
$E$-linearly to these rings.

\begin{defn} An {\it $E$-$B$-pair} of $G_K$ is a couple $W=(W_e, W_\dR^+)$
such that

$\bullet$ $W_e$ is a finite $\BB_{e,E}$-module with a continuous
semilinear action $G_K$-action which is free as a $\BB_e$-module.

$\bullet$ $W^+_\dR\subset W_\dR=\BB_\dR\otimes_{\BB_e} W_e$ is a
$G_K$-stable $\BB^+_{\dR,E}$-lattice.
\end{defn}

By \cite[Remark 1.3]{Naka} $W_e$ is free over $\BB_{e,E}$ and
$W^+_\dR$ is free over $\BB^+_{\dR,E}$.

If $V$ is an $E$-representation of $G_K$, then
$W(V)=(\BB_{e,E}\otimes_E V, \BB^+_{\dR,E}\otimes_E V)$ is an
$E$-$B$-pair, called the {\it $E$-$B$-pair attached to $V$}.

If $S$ is a Banach $E$-algebra, we can define $S$-$B$-pairs
similarly; to each $S$-representation $V$ of $G_K$ is associated an
$S$-$B$-pair $W(V)=(\BB_{e,E}\otimes_E V, \BB^+_{\dR,E}\otimes_E
V)$.

If $W_1=(W_{1,e}, W_{1,\dR}^+)$ and $W_2=(W_{2,e}, W^+_{2,\dR})$ are
two $E$-$B$-pairs, we define $W_1\bigotimes W_2$ to be
$$(W_{1,e}\bigotimes\limits_{\BB_{e,E}}W_{2,e},
W^+_{1,\dR}\bigotimes\limits_{\BB^+_{\dR,E}}W^+_{2,\dR}).$$ Here,
$W_{1,e}\bigotimes\limits_{\BB_{e,E}}W_{2,e}$ is equipped with the
diagonal $G_K$-action, and
$W^+_{1,\dR}\otimes_{\BB^+_{\dR,E}}W^+_{2,\dR}$ is naturally
considered as a $G_K$-stable $\BB^+_{\dR,E}$-lattice of
$$\BB_\dR\otimes_{\BB_e}(W_{1,e}\bigotimes_{\BB_{e,E}}W_{2,e})=
W_{1,\dR}\bigotimes_{\BB_{\dR,E}}W_{2,\dR},$$ where
$W_{1,\dR}=\BB_\dR\otimes_{\BB_e} W_{1,e}$ and
$W_{2,\dR}=\BB_\dR\otimes_{\BB_e} W_{2,e}$.

If $W=(W_e, W^+_\dR)$ is an $E$-$B$-pair with
$W_\dR=\BB_\dR\otimes_{\BB_e} W_e$, we define the dual of $W$ to be
$W^*=(W^*_e, W^{*,+}_\dR)$, where $W^*_e$ is
$\mathrm{Hom}_{\BB_e}(W,\BB_e)$ equipped with the natural
$G_K$-action, and $W^{*,+}_\dR$ is the $G_K$-stable lattice of
$\BB_\dR\otimes_{\BB_e} W^*_e\cong
\mathrm{Hom}_{\BB_\dR}(W_\dR,\BB_\dR)$ defined by
$$ \{\ell\in \mathrm{Hom}_{\BB_\dR}(W_\dR,\BB_\dR): \ell(x)\in \BB^+_\dR \text{ for all } x\in W^+_\dR \}.
$$

The relation between $(\varphi,\Gamma_K)$-modules and $B$-pairs is
built by Berger \cite{Ber08}. We recall Berger's construction below.

Let $M$ be a $(\varphi,\Gamma_K)$-module of rank $d$ over the Robba
ring $\BB^{\dagger}_{\rig,K}$. Berger \cite{Ber08} showed that
$$W_e(M):=(\widetilde{\BB}^\dagger_{\rig}[1/t]\otimes
_{\BB^\dagger_{\rig,K}}M)^{\varphi=1}$$ is a free $\BB_e$-module of
rank $d$ and equipped with a continuous semilinear $G_K$-action.

For sufficiently large $r_0>0$ we can take a unique
$\Gamma_K$-stable finite free $\BB^{\dagger,r}_{\rig,K}$-submodule
$M^{r}\subset M$ such that
$$\BB^{\dagger}_{\rig,K}\otimes_{\BB^{\dagger,r}_{\rig,K}}M^r=M$$ and
$$\mathrm{id}_{\BB^{\dagger, pr}_{\rig,K}}\otimes \varphi_M:
\BB^{\dagger,
pr}_{\rig,K}\otimes_{\BB^{\dagger,r}_{\rig,K}}M^r\xrightarrow{\sim}M^{pr}$$
for any $r\geq r_0$. Berger \cite{Ber08} showed that the
$\BB^+_\dR$-module
$$W^+_\dR(M):=\BB^+_\dR\otimes_{i_n, \BB^{\dagger,
(p-1)p^{n-1}}_{\rig,K}} M^{(p-1)p^{n-1}}$$  is independent of any
$n$ such that $(p-1)p^{n-1}\geq r_0$, and showed that there is a
canonical $G_K$-equivariant isomorphism
$\BB_\dR\otimes_{\BB_e}W_e(M)\xrightarrow{\sim}\BB_\dR\otimes_{\BB^+_\dR}W^+_\dR(M)$.

Put $W(M)=(W_e(M), W^+_\dR(M))$. This is an $E$-$B$-pair of rank
$d=\mathrm{rank}_{\BB^{\dagger}_{\rig,K}}M$.

The following is a variant version of Berger's result \cite[Theorem
2.2.7]{Ber08}.

\begin{prop}\label{thm:berger} \cite[Theorem 1.36]{Naka} The functor $M\mapsto W(M)$ is an exact functor and this
gives an equivalence of categories between the category of
$E$-$(\varphi,\Gamma_K)$-modules and the category of $E$-$B$-pairs
of $G_K$.
\end{prop}

\begin{prop} The functor $M\mapsto W(M)$ respects the tensor
products and duals.
\end{prop}
\begin{proof} Let $M_1$ and $M_2$ be two $E$-$(\varphi,\Gamma_K)$-modules. By taking $\varphi$-invariants, the
isomorphism
$$ (\widetilde{\BB}^\dagger_{\rig}[1/t]\otimes_{\widetilde{\BB}^{\dagger,r}_{\rig,K}}M_1
) \otimes_{\widetilde{\BB}^\dagger_{\rig}\otimes_{\BQ_p}E[1/t]}
(\widetilde{\BB}^\dagger_{\rig}[1/t]\otimes_{\widetilde{\BB}^{\dagger,r}_{\rig,K}}M_2
) \xrightarrow{\sim}
\widetilde{\BB}^\dagger_{\rig}[1/t]\otimes_{\widetilde{\BB}^\dagger_{\rig,K}}(M_1\otimes
M_2) $$ induces a $G_K$-equivariant injective map
$$ W_e(M_1)\otimes_{\BB_{e,E}} W_e(M_2) \rightarrow W_e(M_1\otimes M_2) .
$$ Here, $M_1\otimes
M_2$ denotes the $E$-$(\varphi,\Gamma_K)$-module $
M_1\otimes_{\BB^\dagger_{\rig,K}\otimes_{\BQ_p}E} M_2$. Comparing
dimensions and using \cite[Lemma 1.10]{Naka} we see that this map is
in fact an isomorphism. From the above Berger's construction we see
that the natural map $$
W^+_\dR(M_1)\otimes_{\BB^+_\dR\otimes_{\BQ_p}E}
W^+_\dR(M_2)\rightarrow W^+_\dR(M_1 \otimes M_2)$$ is an
isomorphism. This proves that the functor $M\mapsto W(M)$ respects
tensor products. The proof of that it respects duals is similar.
\end{proof}

\subsection{Semistable $E$-$B$-pairs}

\begin{defn} An {\it $E$-$(\varphi,N)$-module} over $K$ is a
$K_0\otimes_{\BQ_p}E$-module $D$ with a $\varphi\otimes
1$-semilinear isomorphism $\varphi_D: D\rightarrow D$, and a
$K_0\otimes_{\BQ_p}E$-linear map $N_D:D\rightarrow D$ such that
$N_D\varphi_D=p \varphi_D N_D$. A {\it filtered
$E$-$(\varphi,N)$-module} over $K$ is an $E$-$(\varphi,N)$-module
with an exhaustive $\BZ$-indexed descending filtration
$\Fil^\bullet$ on $K\otimes_{K_0}D$.
\end{defn}

We have an isomorphism of rings
\begin{equation} \label{eq:vector-element}
K\otimes_{\BQ_p}E\xrightarrow{\sim} \bigoplus\limits_{\tau\in
\mathrm{Emb}(K,E)}E_\tau, \ \  a\otimes b \mapsto ( \tau(a)b
)_{\tau}, \end{equation} where $E_\tau$ is a copy of $E$ for each
$\tau\in \mathrm{Emb}(K,E)$. Let $e_\tau$ be the unity of $E_\tau$.
Then $1=\sum_{\tau}e_\tau$. Put $D_\tau= e_\tau (K\otimes_{K_0}D)$.
Then
$K\otimes_{K_0}D=\bigoplus\limits_{\tau\in\mathrm{Emb}(K,E)}D_\tau$.
Let $\Fil_\tau$ denote the induced filtration on $D_\tau$.

\begin{defn} Let $W=(W_e, W^+_\dR)$ be an $E$-$B$-pair. We define
$\BD_\cris(W)=(\BB_{\mathrm{max}}\otimes_{\BB_e}W_e)^{G_K}$,
$\BD_\st(W)=(\BB_\log\otimes_{\BB_e}W_e)^{G_K}$ and
$\BD_\dR(W)=(\BB_\dR\otimes_{\BB_e}W_e)^{G_K}$.  Then we have
$\dim_{K_0}(\BD_?(W))\leq \mathrm{rank}_{\BB_e}W_e$ for $?=\cris,
\st$, and $\dim_{K}(\BD_\dR(W))\leq \mathrm{rank}_{\BB_e}W_e$. We
say that $W$ is {\it crystalline} (resp. {\it semistable}) if
$\dim_{K_0}(\BD_?(W)):= \mathrm{rank}_{\BB_e}W_e$ for $?=\cris$
(resp. $\st$).
\end{defn}

If $W$ is a semistable $E$-$B$-pair, we attach to $W$ a filtered
$E$-$(\varphi,N)$-module as follows. The underlying
$E$-$(\varphi,N)$-module is $\BD_\st(W)$; the filtration on
$\BD_\dR(W)=K\otimes_{K_0}\BD_\st(W)$ is given by $\Fil^i
\BD_\dR(W)= t^i W^+_\dR\cap \BD_\dR(W)$.

\begin{prop}\label{prop:berger}
\begin{enumerate}
\item\label{it:berger-a} The functor $W\mapsto \BD_\st(W)$ realizes an equivalence of
categories between the category of semistable $E$-$B$-pairs of $G_K$
and the category of filtered $E$-$(\varphi,N)$-modules over $K$.
\item\label{it:berger-b-2} If $W_1$ and $W_2$ are semistable, then
so is $W_1\otimes W_2$.
\item\label{it:berger-d} The functor $W\mapsto \BD_\st(W)$ respects the tensor products
and duals.
\item\label{it:berger-b} If
\[\xymatrix{ 0 \ar[r] & W_1 \ar[r] & W \ar[r] & W_2 \ar[r] & 0 }\]
is a short exact sequence of $E$-$B$-pairs, and $W$ is semistable,
then $W_1$ and $W_2$ are semistable.
\item\label{it:berger-c} The functor $W\mapsto \BD_\st(W)$ is exact.
\end{enumerate}
\end{prop}
\begin{proof} Assertion (\ref{it:berger-a}) follows from \cite[Proposition 2.3.4]{Ber08}. See also \cite[Theorem 1.18
(2)]{Naka}.

Let $W_1$ and $ W_2$ be two $E$-$B$-pairs. The isomorphism
$$ (\BB_{\log} \otimes_{\BB_e} W_1)
\otimes_{\BB_\log\otimes_{\BQ_p}E} (\BB_{\log} \otimes_{\BB_e} W_2)
\xrightarrow{\sim} \BB_{\log} \otimes_{\BB_e} (W_1 \otimes W_2)
$$ induces an injective map
\begin{equation}\label{eq:inj-iso}
\BD_\st(W_1)\otimes_{K_0\otimes_{\BQ_p}E} \BD_\st(W_2) \rightarrow
\BD_\st(W_1\otimes W_2).\end{equation} When $W_1$ and $W_2$ are
semistable, the dimension of the source over $K_0$ is
$\frac{\mathrm{rank}_{\BB_e}W_1
\mathrm{rank}_{\BB_e}W_2}{[E:\BQ_p]}$. The dimension of the target
over $K_0$ is always  equal to or less than  $\mathrm{rank}_{\BB_e}
(W_1\otimes W_2)=\frac{\mathrm{rank}_{\BB_e}W_1
\mathrm{rank}_{\BB_e}W_2}{[E:\BQ_p]}$. Hence, (\ref{eq:inj-iso}) is
an isomorphism, and so $W_1\otimes W_2$ is semistable. This proves
(\ref{it:berger-b-2}). Similarly, the isomorphism
\begin{equation} \label{eq:iso-deRham}
(\BB_{\dR} \otimes_{\BB_e} W_1) \otimes_{\BB_\dR\otimes_{\BQ_p}E}
(\BB_{\dR} \otimes_{\BB_e} W_2) \xrightarrow{\sim} \BB_{\dR}
\otimes_{\BB_e} (W_1 \otimes W_2)
\end{equation} induces an isomorphism
$$\BD_\dR(W_1)\otimes_{K\otimes_{\BQ_p}E} \BD_\dR(W_2) \rightarrow
\BD_\dR(W_1\otimes W_2).$$ Via the isomorphism (\ref{eq:iso-deRham})
the filtration on $(\BB_{\dR} \otimes_{\BB_e} W_1)
\otimes_{\BB_\dR\otimes_{\BQ_p}E} (\BB_{\dR} \otimes_{\BB_e} W_2)$
coincides with that on $\BB_{\dR} \otimes_{\BB_e} (W_1 \otimes W_2)
$. Therefore, the filtration on
$\BD_\dR(W_1)\otimes_{K\otimes_{\BQ_p}E} \BD_\dR(W_2) $ and that on
$\BD_\dR(W_1\otimes W_2)$ coincide. Indeed, they are the
restrictions of the filtrations on $(\BB_{\dR} \otimes_{\BB_e} W_1)
\otimes_{\BB_\dR\otimes_{\BQ_p}E} (\BB_{\dR} \otimes_{\BB_e} W_2)$
and $\BB_{\dR} \otimes_{\BB_e} (W_1 \otimes W_2)$ respectively.
Similarly we can show that $W\mapsto \BD_\st(W)$ respects duals.
This proves (\ref{it:berger-d}).

For (\ref{it:berger-b}) we have the following exact sequence
\begin{equation}\label{eq:exact-3-W} \xymatrix{ 0 \ar[r] & \BD_\st(W_1) \ar[r] & \BD_\st(W)
\ar[r] & \BD_\st(W_2)  .} \end{equation} So (\ref{it:berger-b})
follows from a dimension argument. Furthermore, when $W$ is
semistable,  $\BD_\st(W)\rightarrow \BD_\st(W_2)$ is surjective. For
any $i\in\BZ$ we write $d_i(W)$ for $\dim_{K}\Fil^i \BD_\st(W)$. As
the maps in the exact sequence (\ref{eq:exact-3-W}) respect
filtrations, we have $d_i(W)\leq d_i(W_1)+ d_i(W_2)$. Similarly, we
have $d_{1-i}(W^*)\leq d_{1-i}(W^*_1)+ d_{1-i}(W^*_2)$. As $W\mapsto
\BD_\st(W)$ respects duals, we have $d_i(W)=\dim_{K}(\BD_\dR(W))-
d_{1-i}(W^*)$. Then
\begin{eqnarray*} d_i(W) & = & \dim_{K}(\BD_\dR(W))- d_{1-i}(W^*) \\ & \geq
& (\dim_{K}(\BD_\dR(W_1))- d_{1-i}(W_1^*) ) +
\dim_{K}(\BD_\dR(W_2))- d_{1-i}(W_2^*) \\ &=&
d_i(W_1)+d_i(W_2).\end{eqnarray*} Thus we must have
$d_i(W)=d_i(W_1)+d_i(W_2)$ for all $i\in\BZ$. In other words, the
maps in (\ref{eq:exact-3-W}) are strict for the filtrations, which
shows (\ref{it:berger-c}).
\end{proof}

By \cite[Proposition 2.3.4]{Ber08} the quasi-inverse of the functor
$\BD_\st$ 
is given by
\begin{equation} \label{eq:def-DB}
\BD_B(D)=((\BB_{\log}\otimes_{K_0 }D)^{\varphi=1, N=0},
\Fil^0(\BB_{\dR }\otimes_{K_0 }D)).\end{equation}

For a filtered $E$-$(\varphi,N)$-module $D$ we put $$ \BX_\log(D)
=(\BB_\log\otimes_{K_0}D)^{\varphi=1, N=0} \  \ \text{and} \ \
\BX_\dR(D)= \BB_\dR \otimes_{K_0}D  / \Fil^0( \BB_\dR \otimes_{K_0}D
).
$$ If $\BD_B(D)=(W_e, W^+_\dR)$, then
$\BX_\log(D)=W_e$ and
$\BX_\dR(D)=(\BB_\dR\otimes_{\BB_e}W_e)/W^+_\dR$.


\section{$S$-$B$-pairs of rank $1$ and triangulations}
\label{sec:family-B-pairs}

\subsection{$S$-$B$-pairs of rank $1$}

Let $S$ be a Banach $E$-algebra.

For any $a\in S^\times$ we define a filtered $S$-$\varphi$-module
$D_a$ as follows. As a $K_0\otimes_{\BQ_p}S$-module,
$$D_a=K_0\otimes_{\BQ_p}S=\oplus_{\tau: K_0\hookrightarrow E}S
e_\tau;$$ the $\varphi\otimes 1$-semilinear action $\varphi$ on
$D_a$ satisfies
$$\varphi(e_{\mathrm{id}}) = e_{\varphi^{-1}}, \ \varphi(e_{\varphi^{-1}}) = e_{\varphi^{-2}}, \ \cdots, \ \varphi(e_{\varphi^{1-f}}) = a e_{\mathrm{id}};$$ the
descending filtration on $D_{a,K}=K\otimes_{\BQ_p}S$ is given by
$\Fil^0 D_{a,K}=D_{a,K}$ and $\Fil^1 D_{a,K}=0$.

\begin{lem} \label{lem:kisin} If $a\in S$ satisfies that $a-1$ is topologically nilpotent,
then there exists a unit $u_0\in \BB_\max\widehat{\otimes}_{K_0}S$
such that $\varphi^{[K_0:\BQ_p]}(u_0)= au_0$. Consequently
$$ \{x\in \BB_\max\widehat{\otimes}_{K_0}S: \varphi^{[K_0:\BQ_p]}(x)= ax\} = (\BB_{e,K_0}\widehat{\otimes}_{K_0}S) u_0 . $$
\end{lem}
\begin{proof} Let $\BQ_p^{\mathrm{ur}}$ be the completed unramified extension
of $\BQ_p$. Then there exists an inclusion
$\BQ_p^{\mathrm{ur}}\hookrightarrow \BB_\max$ that is compatible
with $\varphi$.

As $\varphi^{[K_0:\BQ_p]}-1$ is surjective on $\BQ_p^{\mathrm{ur}}$,
there exists a sequence $c_0=1, c_1, \cdots$ of elements in
$\BQ_p^{\mathrm{ur}}$ such that
$$(\varphi^{[K_0:\BQ_p]}-1)c_i=c_{i-1}$$ for $i\geq 1$. The image of
$c_i$ by the map
$$\BQ_p^{\mathrm{ur}}\hookrightarrow\BB_\max\rightarrow
\BB_\max\widehat{\otimes}_{K_0}S$$ is again denoted by $c_i$. Put
$$u_0=\sum_{i=0}^{\infty}c_i(a-1)^i.$$ Then $u_0$ is a unit and we have
$\varphi^{[K_0:\BQ_p]}u_0=au_0$.
\end{proof}

\begin{prop}\label{prop:S-B-pair} If $a\in S$ satisfies that $a-1$ is topologically
nilpotent, then $\BD_B(D_a)$ is an $S$-$B$-pair of rank $1$. Here
$\BD_B$ is the functor defined by $($\ref{eq:def-DB}$)$.
\end{prop}
\begin{proof} For each $z\in \BB_\max\widehat{\otimes}_{\BQ_p}D_a$ we write $z=\sum c_\tau
e_\tau$ with $c_\tau\in \BB_\max\widehat{\otimes}_{K_0,\tau}S$. Then
$\varphi(z)=z$ if and only if
$\varphi(c_{\varphi^i})=c_{\varphi^{i-1}}$ ($i=1,\cdots,
[K_0:\BQ_p]$) and
$\varphi^{[K_0:\BQ_p]}(c_{\mathrm{id}})=ac_{\mathrm{id}}$. Our
assertion follows from Lemma \ref{lem:kisin}.
\end{proof}

For any $a\in S^\times$, let $\delta_a: K^\times \rightarrow
S^\times$ denote the character such that $\delta_a(\pi_K)=a$ and
$\delta_a|_{\mathfrak{o}_K^\times}=1$.

\begin{rem} In the case of $S=E$, for any $u\in E^\times$,  $\BD_B(D_u)$ coincides with the $E$-$B$-pair
$W(\delta_u)$ defined in \cite{Naka} (see \cite[\S 1.4]{Naka}). From
now on the base change of $W(\delta_u)$ from $E$ to $S$ is again
denoted by $W(\delta_u)$.
\end{rem}

Let $\delta: K^\times \rightarrow S^\times$ be a continuous
character such that $\delta(\pi_K)$ is of the form
$\delta(\pi_K)=au$, where $u\in E^\times$ and $a\in S$ satisfies
that $a-1$ is topologically nilpotent. We call such a character {\it
a good character}. Let $W_a$ be the resulting $S$-$B$-pair in
Proposition \ref{prop:S-B-pair}. Let $\delta'$ be the unitary
continuous character $K^\times\rightarrow E^\times$ such that
$\delta'|_{\mathfrak{o}^\times_K}=\delta|_{\mathfrak{o}^\times_K}$
and $\delta'(\pi_K)=1$. By local class field theory, this induces a
continuous character $\widetilde{\delta}': G_K\rightarrow S^\times$
such that $\widetilde{\delta}'\circ \mathrm{rec}_K=\delta'$. Then we
put $$ W(\delta)=W(S(\widetilde{\delta}'))\otimes W(\delta_u)
\otimes W_a, $$ where $W(S(\widetilde{\delta}'))$ is the
$S$-$B$-pair attached to the Galois representation
$S(\widetilde{\delta}')$.

If $\delta$ is a continuous character $\delta: K^\times\rightarrow
S^\times$, we write $\log(\delta)$ for the logarithmic of
$\delta|_{\mathfrak{o}_K^\times}$, which is a $\BZ_p$-linear
homomorphism $\log(\delta): K \rightarrow S$.

For any $\tau\in\mathrm{Emb}(K,E)$ we use the same notation $\tau$
to denote the composition of $\tau: K\hookrightarrow E$ and $E
\hookrightarrow S$. Then $\{\tau: K\hookrightarrow S\}$ is a basis
of $\Hom_{\BZ_p}(E, S)$ over $S$. Write
$\log(\delta)=\sum_{\tau}k_\tau \tau$, $k_\tau\in S$. We call
$(k_\tau)_{\tau}$ the {\it weight vector} of $\delta$ and denote it
by $\vec{w}(\delta)$. We use $w_\tau(\delta)$ to denote $k_\tau$.

\begin{rem} Let $S$ be an affinoid algebra over $E$. For any continuous character $\delta: K^\times
\rightarrow S^\times$ and any point $z_0$ of $\Max(S)$, there exists
an affinoid neighborhood $U=\Max(S')$ of $z_0$ in $\Max(S)$ such
that the restriction of $\delta$ to $U$ is good.
\end{rem}

\begin{lem}\label{lem:varph-gamma-fil} Let $\delta$ be a character of $K^\times$ with values in
$S=E[Z]/(Z^2)$, $\bar{\delta}$ the character of $K^\times$ with
values in $E$ obtained from $\delta$ modulo $(Z)$. Write
$\delta=\bar{\delta}_S(1+Z\epsilon)$, where $\bar{\delta}_S$ is the
character $K^\times \xrightarrow{\bar{\delta}} E^\times
\hookrightarrow S^\times$. Let $\epsilon'$ be the additive character
of $G_K$ such that $\epsilon' \circ\mathrm{rec}_K(p)=0$ and
$\epsilon'\circ\mathrm{rec}_K|_{\fo_K^\times}=\epsilon|_{\fo_K^\times}$.

Assume that $W(\bar{\delta})$ is crystalline and
$\varphi^{[K_0:\BQ_p]}$ acts on $\BD_\cris(W(\bar{\delta}))$ by
$\alpha$. Then there is a nonzero element $$ x\in
(\BB_{\max,E}\otimes_{\BB_{e,E}}W(\delta)_e)^{\varphi^{[K_0:\BQ_p]}=\alpha
(1+Zv_p(\pi_K)\epsilon(p)), G_K= (1+Z\epsilon')} $$ whose reduction
modulo $Z$ is a basis of $\BD_\st(W(\bar{\delta}))$ over
$K\otimes_{\BQ_p}E$.
\end{lem}
\begin{proof} This follows from the fact that
$W(\delta)= W(\bar{\delta}_S)\otimes
W_{\delta_{1+Zv_p(\pi_K)\epsilon(p)}}\otimes W(1+Z\epsilon') .$
\end{proof}

\subsection{Triangulations and refinements}

Now let $S$ be an affinoid algebra over $E$. For any open affinoid
subset $U$ of $S$ and any $S$-$B$-pair $W$ let $W_U$ denote the
restriction to $U$ of $W$.

\begin{defn} Let $W$ be an $S$-$B$-pair of rank $n$, $z_0$ a point of $\Max(S)$. If there
is
\begin{quote}
$\bullet$ an affinoid neighborhood $U=\Max(S_U)$ of $z_0$,

$\bullet$ a strictly increasing filtration $$ \{0\}= \Fil_0 W_U
\subset \Fil_1 W_U \subset\cdots \subset \Fil_n W_U = W_U $$ of
saturated free sub-$S_U$-$B$-pairs, and

$\bullet$ $n$ good continuous characters $\delta_i:
\BQ_p^\times\rightarrow S_U^\times$ 

\end{quote}
\noindent such that for any $i=1,\cdots, n$,
$$\Fil_i W_U/\Fil_{i-1} W_U \simeq W(\delta_i),$$ 
we say that $W$ is {\it locally triangulable at $z_0$};  we call
$\Fil_\bullet$ a {\it local triangulation} of $W$ at $z_0$, and call
$(\delta_1,\cdots, \delta_n)$ the {\it local triangulation
parameters} attached to $\Fil_\bullet$.
\end{defn}

Please consult \cite{Cz2010,Be2011} for more knowledge on
triangulations.

To discuss the relation between triangulations and refinements, we
restrict ourselves to the case of $S=E$.

Let $D$ be a filtered $E$-$(\varphi,N)$-module of rank $n$. The
operator $\varphi^{[K_0:\BQ_p]}$ on $D$ is
$K_0\otimes_{\BQ_p}E$-linear. We assume that the eigenvalues of
$\varphi^{[K_0:\BQ_p]}:D\rightarrow D$ are all in
$K_0\otimes_{\BQ_p}E$, i.e. there exists a basis of $D$ over
$K_0\otimes_{\BQ_p}E$ such the matrix of $\varphi^{[K_0:\BQ_p]}$
with respect to this basis is upper-triangular.

Following Mazur \cite{Mazur} we define a {\it refinement} of $D$ to
be a filtration on $D$
$$ 0 = \CF_0 D \subset \CF_1 D \subset \cdots \subset \CF_n D= D $$ by
$E$-subspaces stable by $\varphi_D$ and $N_D$, such that each factor
$\mathrm{gr}^\CF_i D= \CF_i D/\CF_{i-1} D$ ($i=1,\cdots, n$) is of
rank $1$ over $K_0\otimes_{\BQ_p}E$. Any refinement fixes an
ordering $\alpha_1,\cdots, \alpha_n$ of eigenvalues of
$\varphi^{[K_0:\BQ_p]}$ and an ordering $\vec{k}_1,\cdots,
\vec{k}_n$ of Hodge-Tate weights of $K\otimes_{K_0}D$ taken with
multiplicities such that the eigenvalue of $\varphi^{[K_0:\BQ_p]}$
on $\mathrm{gr}^\CF_i D$ is $\alpha_i$ and the Hodge-Tate weight of
$\mathrm{gr}^\CF_i\CD$ is $\vec{k}_i$.

We have the following analogue of \cite[Proposition 1.3.2]{Ben}.

\begin{prop} Let $W$ be a semistable $E$-$B$-pair,
$D=\BD_\st(W)$.
\begin{enumerate}
\item\label{it:trian-refine-a}
The equivalence of categories between the category of semistable
$E$-$B$-pairs and the category of filtered $E$-$(\varphi,N)$-modules
induces a bijection between the set of triangulations on $W$ and the
set of refinements on $D$.
\item\label{it:trian-refine-b}
If $(\Fil_i W)$ is a triangulation of $W$ with triangulation
parameters $(\delta_1,\cdots, \delta_n)$ that correspond to a
refinement $\CF_\bullet D$ of $D$ with the ordering of Hodge-Tate
weights being $\vec{k}_1, \cdots, \vec{k}_n$, then
$\delta_i=\tilde{\delta}_i\prod\limits_{\tau\in\mathrm{Emb}(K,E)}\tau(x)^{k_{i,\tau}}$,
where $\tilde{\delta}_i$ is a smooth character.
\end{enumerate}
\end{prop}
\begin{proof} Assertion (\ref{it:trian-refine-a}) follows from the
fact that $\BD_\st$ is an exact. Assertion (\ref{it:trian-refine-b})
follows from \cite[Lemma 4.1]{Naka}.
\end{proof}

\section{Cohomology Theory} \label{sec:coh}
\subsection{Cohomology of $(\varphi,\Gamma_K)$-modules and cohomology of
$B$-pairs} \label{sec:coh-B-pair}

Let $M$ be a $(\varphi,\Gamma_K)$-module. Assume that $\Gamma_K$ has
a topological generator $\gamma$. Define the cohomology
$H^\bullet_\PGam(M)$ by the complex $C^\bullet(M)$ defined by
\[ C^0(M)=M \xrightarrow{(\gamma-1, \varphi-1)} C^1(M)=M\oplus M \rightarrow C^2(M)=M,
\] where the map $C^1(M) \rightarrow C^2(M)$ is given by $(x,y)\mapsto
(\varphi-1)x-(\gamma-1)y$. Denote the kernel of $C^1(M)\rightarrow
C^2(M)$ by $Z^1(M)$.

There is a one-to-one correspondence between $H^1(M)$ and the set of
extensions of $M_0$ by $M$ in the category of
$(\varphi,\Gamma_K)$-modules, where $M_0=\BB^{\dagger}_{\rig,K}e_0$
is the trivial $(\varphi,\Gamma_K)$-module with
$\varphi(e_0)=\gamma(e_0)=e_0$. Let $\tilde{M}$ be an extension of
$M_0$ by $M$, and let $\tilde{e}$ be any lifting of $e_0$ in
$\tilde{M}$. Then the element in $H^1(M)$ corresponding to the
extension $\tilde{M}$ is the class of $((\gamma-1)\tilde{e},
(\varphi-1)\tilde{e})\in Z^1(M)$.

In \cite{Naka} Nakamura introduced a cohomology for $B$-pairs and
use it to compute the cohomology of $(\varphi,\Gamma_K)$-modules.

If $W=(W_e, W^+_\dR)$ is an $E$-$B$-pair, let $C^\bullet(W)$ be the
complex of $G_K$-modules defined by
$$ C^0(W):=W_e \rightarrow C^1(W):= W_\dR/W^+_\dR.
$$ Here, $W_e\rightarrow W_\dR/W^+_\dR$ is the natural map.

\begin{defn} Let $W=(W_e, W^+_\dR)$ be an $E$-$B$-pair. We
define the {\it Galois cohomology of $W$} by $H^i_B(W):= H^i(G_K,
C^\bullet(W))$.
\end{defn}

By definition there is a long exact sequence \begin{equation}
\label{eq:long-exact-sq}
\cdots \rightarrow H^i_B(W) \rightarrow
H^i(G_K, W_e) \rightarrow H^i(G_K, W_\dR/W^+_\dR) \rightarrow \cdots
.
\end{equation}

For a $G_K$-module $M$ put $C^0(M)=M$ and let $C^i(M)$ be the space
of continuous functions from $(G_K)^{\times i}$ to $M$. Let
$\delta_0: C^0(M)\rightarrow C^1(M)$ be the map $x\mapsto (g \mapsto
g(x)-x)$ and let $\delta_1: C^1(M)\rightarrow C^2(M)$ be the map $f
\mapsto ((g_1,g_2)\mapsto f(g_1g_2)-f(g_1)-g_1f(g_2))$.

Nakamura \cite{Naka} showed that $H^1_B(W)$ is isomorphic to
$\ker(\tilde{\delta}_1)/\mathrm{im}(\tilde{\delta}_0)$, where
$\tilde{\delta_0}$ and $\tilde{\delta}_1$ are defined by
\begin{eqnarray*} \tilde{\delta}_0: C^0(W_e)\oplus C^0(W^+_\dR)
& \rightarrow & C^1(W_e)\oplus C^1(W^+_\dR)\oplus C^0(W_\dR) :\\
(x,y) & \mapsto & (\delta_0(x), \delta_0(y), x-y), \\
\tilde{\delta}_1: C^1(W_e)\oplus C^1(W^+_\dR)\oplus C^0(W_\dR)
& \rightarrow &  C^2(W_e)\oplus C^2(W^+_\dR)\oplus C^1(W_\dR):\\
(f_1,f_2, x) & \mapsto & (\delta_1(f_1), \delta_1(f_2),
f_1-f_2-\delta_0(x)).
\end{eqnarray*}
The map $H^1_B(W)\rightarrow H^1(G_K, W_e)$ is induced by the
forgetful map $$ C^1(W_e)\oplus C^1(W^+_\dR)\oplus
C^0(W_\dR)\rightarrow C^1(W_e). $$

There is a one-to-one correspondence between $H^1(G_K,W)$ and the
set of extensions of $W_0$ by $W$ in the category of $E$-$B$-pairs.
Here, $W_0=(\BB_e\otimes_{\BQ_p}E, \BB^+_{\dR}\otimes_{\BQ_p}E)$ is
the trivial $E$-$B$-pair. Let $\tilde{W}=(\tilde{W}_e,
\tilde{W}^+_\dR)$ be an extension of $W_0$ by $W$. Let
$(\tilde{w}_e, \tilde{w}^+_\dR)$ be a lifting in $\tilde{W}$ of
$(1,1)\in W_0$. Then the element in $H^1_B(W)$ corresponding to the
extension $\tilde{W}$ is just the class of $( (\sigma\mapsto
(\sigma-1)\tilde{w}_e), (\sigma\mapsto (\sigma-1)\tilde{w}^+_\dR),
\tilde{w}_e-\tilde{w}^+_\dR)\in \ker(\tilde{\delta}_1)$.

By Proposition \ref{thm:berger} there is a one-to-one correspondence
between $\Ext(M_0,M)$ and $\Ext(W_0, W(M))$. It induces a natrual
isomorphism
$$ i_M: H^1_{\PGam}(M) \rightarrow H^1_B( W(M)). $$

\subsection{$1$-cocycles from infinitesimal
deformations} \label{ss:cocycle-inform}

Let $S$ be the $E$-algebra $E[Z]/(Z^2)$, $\tilde{M}$ an
$S$-$(\varphi,\Gamma_K)$-module. Let $\{e_1, \cdots, e_n\} $ be an
$S$-basis of $\tilde{M}$, $\{e^*_1, \cdots, e^*_n\}$ the dual basis
of $\tilde{M}^*$. Put $M=\tilde{M}\otimes_S E$ and
$M^*=\tilde{M}^*\otimes_SE$. Let $e_{i,z}$ denote $e_i \ \mathrm{
mod } \ Z$, and $e^*_{j,z} $ denote $e^*_j \ \mathrm{ mod } \ Z$.
Then $\{e_{1,z},\cdots, e_{n,z}\}$ is an $E$-basis of $M$, and
$\{e^*_{1,z},\cdots, e^*_{n,z}\}$ is an $E$-basis of $M^*$.

The matrices of $\varphi$ and $\gamma$ with respect to
$\{e_1,\cdots, e_n\}$ are denote by $\tilde{A}_\varphi$ and
$\tilde{A}_\gamma$ respectively, so that $ \varphi(e_j) = \sum_{i}
(\tilde{A}_{\varphi})_{ij} e_i$ and $\gamma (e_j)=\sum_i
(\tilde{A}_\gamma)_{ij} e_i$. Write $\tilde{A}_\varphi = (I_n+
ZU_\varphi)A_\varphi$ and $\tilde{A}_\gamma = (I_n+
ZU_\gamma)A_\gamma$. Put
$$  c_{\PGam}(\tilde{M})=( \sum\limits_{i,j} (U_{\varphi})_{ i,j} e^*_{j,z}\otimes e_{i,z}  , \sum\limits_{i,j}
(U_{\gamma})_{ i,j} e^*_{j,z}\otimes e_{i,z} ).$$

Write $\BD_B(\tilde{M})=(\tilde{W}_e, \tilde{W}^+_\dR)$,
$\BD_B(M)=W$ and $\BD_B(M^*)=W^*$.

Let $f_1, \cdots, f_n$ be a basis of $\tilde{W}_e$ over $\BB_{e,E}$,
and let
 $g_1,\cdots, g_n$ be a basis of
$\tilde{W}^+_\dR$ over $\BB^+_{\dR,E}$. We write the matrix of
$\sigma\in G_K$ with respect to the basis $\{f_1, \cdots, f_n\}$ by
$(I_n + Z U_{e,\sigma})A_{e,\sigma}$, and the matrix of $\sigma$
with respect to the basis $\{g_1, \cdots, g_n\}$ by $(I_n + Z
U^+_{\dR,\sigma})A^+_{\dR,\sigma}$. Here, $$ U_{e, \sigma} \in
\mathrm{M}_{n }(\BB_{e,E}), \ U^+_{\dR,\sigma}\in \mathrm{
M}_{n}(\BB^+_{\dR,E}), \ A_{e, \sigma} \in \GL_{n}(\BB_{e,E}), \
\text{and} \  A^+_{\dR,\sigma}\in \GL_{n}(\BB^+_{\dR,E}). $$ Write
$(f_1, \cdots, f_n) = (g_1, \cdots, g_n) (I_n +Z U_\dR)A_\dR$ and
put {\small
$$c_{B}(\tilde{M}) =\Big( (\sigma\mapsto \sum_{i,j}(U_{e, \sigma})_{ij}f_{j,z}^*\otimes f_{i,z}),
(\sigma\mapsto \sum_{i,j}(U^+_{\dR, \sigma})_{ij}g_{j,z}^*\otimes
g_{i,z}),  \sum_{i,j}(U_{\dR})_{ij}g_{j,z}^*\otimes g_{i,z} \Big).
$$ }

\begin{prop}
\begin{enumerate}
\item\label{it:PGam-B-a} $c_\PGam(\tilde{M})$ is in $Z^1(M^*\otimes
M)$.
\item\label{it:PGam-B-b} $c_B(\tilde{M})$ is in $\ker(\tilde{\delta}_{1, W^*\otimes W})$.
\item\label{it:PGam-B-c} We have $i_M([c_\PGam (\tilde{M})])=[c_B(\tilde{M})]$.
\end{enumerate}
\end{prop}
\begin{proof} It is easy to verify (\ref{it:PGam-B-a}) and (\ref{it:PGam-B-b}).

Put $M^*_S=M^*\otimes_ES$. We consider $M^*_S\otimes_S \tilde{M}$ as
an extension of $M^*\otimes_E M$ by itself, and form the following
commutative diagram
$$ \xymatrix{ && & M_0\ar[d] \\ 0 \ar[r] & M^*\otimes_E M \ar[r] & M^*_S\otimes_S \tilde{M} \ar[r] & M^*\otimes_E M \ar[r] & 0,}
$$ where the vertical map $M_0\rightarrow M^*\otimes_E M$ is given
by $1\mapsto \sum_{i=1}^n e^*_{i,z}\otimes e_{i,z}$, which does not
depend on the choice of the basis $\{e_1,\cdots, e_n\}$. Pulling
back $M^*_{S}\otimes_S \tilde{M}$ via $M_0\rightarrow M^*\otimes_E
M$ we obtain an extension of $M_0$ by $M^*\otimes_E M$. Let
$\mathcal{M}$ denote the resulting extension. Then $\mathcal{M}$ is
a sub-$E$-$B$-pair of $M^*_S\otimes_S \tilde{M}$. Put
$\BD_B(\mathcal{M})=(\mathcal{W}_e, \mathcal{W}^+_\dR)$.

A lifting of $1$ in $\mathcal{W}_e$ is $\sum_j f^*_{j,z}\otimes
f_j$, and a lifting of $1$ in $\mathcal{W}^+_\dR$ is $\sum_j
g^*_{j,z}\otimes g_j$. We have \begin{eqnarray*} (\sigma-1)\sum_j
f^*_{j,z}\otimes f_j &=& \sigma (f^*_{1,z},\cdots, f^*_{n,z})
\otimes \sigma \left(
\begin{array}{c} f_1 \\ f_2 \\ \vdots \\ f_n \end{array}\right) - (f^*_{1,z},\cdots, f^*_{n,z})
\otimes \left(
\begin{array}{c} f_1 \\ f_2 \\ \vdots \\ f_n \end{array}\right)
\\
&=& (f^*_{1,z},\cdots, f^*_{n,z}) (A_{e,\sigma}^t)^{-1} \otimes
A_{e,\sigma}^t(1+zU_{e,\sigma}^t) \\
&=& (f^*_{1,z},\cdots, f^*_{n,z}) \otimes U_{e,\sigma}^t z \left(
\begin{array}{c} f_1 \\ f_2 \\ \vdots \\ f_n \end{array}\right).
\end{eqnarray*} Similarly,
\begin{eqnarray*} (\sigma-1)\sum_j
g^*_{j,z}\otimes g_j = (g^*_{1,z},\cdots, g^*_{n,z}) \otimes
(U^+_{\dR,\sigma})^t z \left(
\begin{array}{c} g_1 \\ g_2 \\ \vdots \\ g_n \end{array}\right),
\end{eqnarray*}
and
$$ \sum_j
f^*_{j,z}\otimes f_j - \sum_j g^*_{j,z}\otimes g_j =
(g^*_{1,z},\cdots, g^*_{n,z}) \otimes U_{\dR}^t z \left(
\begin{array}{c} g_1 \\ g_2 \\ \vdots \\ g_n \end{array}\right). $$
Hence the element in $H^1_B(\BD_B(M^*\otimes_EM))$ attached to the
extension $\BD_B(\mathcal{M})$ is $[c_B(\tilde{M})]$.

A similar computation shows that the element in
$H^1_\PGam(M^*\otimes_E M)$ attached to the extension $\mathcal{M}$
is $[c_\PGam(\tilde{M})]$. Now (\ref{it:PGam-B-c}) follows.
\end{proof}

\section{The reciprocity law and an application} \label{sec:aux-for}

\subsection{Reciprocity law}

In \cite[Section 2]{Zhang} using local class field theory Zhang
precisely described the perfect pairing
$$ H^1(G_K, E) \times H^1(G_K,E(1)) \rightarrow H^2(G_K, E(1)). $$
We recall it below.

The Kummer theory gives us a canonical isomorphism so called the
Kummer map
\begin{eqnarray*}
\lim_{\overleftarrow{\;\;
n\;\:}}(K^\times/(K^\times)^{p^n})\otimes_{\BZ_p} E & \rightarrow &
H^1(G_K, E(1)) \\ \sum_i \alpha_i\otimes a_i & \mapsto & \sum_i a_i
[(\alpha_i)].
\end{eqnarray*} Here $(\alpha)$ is the $1$-cocycle such that
$$ \frac{g (\sqrt[p^n]{\alpha})}{\alpha} = \varepsilon_n^{(\alpha_g)}
$$ for $\alpha\in K^\times$ and $g\in G_K$, where
$(\sqrt[p^{n+1}]{\alpha})^p=\sqrt[p^n]{\alpha}$. Combining the
Kummer map and the exponent map $$\exp: p\fo_K\rightarrow K^\times$$
and extending it by linearity we obtain an embedding from
$K\otimes_{\BQ_p}E$ to $H^1(G_K, E(1))$, again denoted by $\exp$.
Then we have $$ H^1(G_K, E(1))=\exp(K\otimes_{\BQ_p}E)\oplus E \cdot
[(p)].$$

Let $\Hom(G_K,E)$ be the group of additive characters of $G_K$ with
values in $E$. As the action of $G_K$ on $E$ is trivial,
$H^1(G_K,E)$ is naturally isomorphic to $\Hom(G_K, E)$. Let
$\psi_0:G_K\rightarrow E$ be the additive character that vanishes on
the inertial subgroup of $G_K$ and maps the geometrical Frobenius to
$[K_0:\BQ_p]$. For any $\tau\in \mathrm{Emb}(K,E)$ let $\psi_\tau$
be the composition $\tau\circ \log \circ \mathrm{rec}_K^{-1}$
\footnote{ Since the character $\psi_\tau$ of the Weil group $W_K$
sends any lifting of the $q$th power Frobenius to $0$, it can be
extended to a character of $G_K$ which is again denoted by
$\psi_\tau$}, where $\log$ is normalized such that $\log (p)=0$.
Then $\{\psi_0, \psi_\tau : \tau\in \mathrm{Emb}(K,E) \}$ is an
$E$-basis of $H^1(G_K,E)$.

\begin{lem} \label{lem:zhang}
\cite[Proposition 2.1]{Zhang} The cup product of $a_0 \psi_0 +
\sum_{\tau\in \mathrm{Emb}(K,E)} a_\tau \psi_\tau $ $(a_0, a_\tau\in
E)$ and $b_0[(p)]+  \exp(b)$ $(b_0\in E, b\in K\otimes_{\BQ_p} E)$
is
$$\Big( a_0 b_0 - \mathrm{tr}_{K/\BQ_p} ((a_\tau)_\tau \cdot b) \Big) (\psi_0\cup [(p)]). $$
Here,  $(a_\tau)_{\tau}$ is considered as an element in
$K\otimes_{\BQ_p}E$ via the isomorphism $(\ref{eq:vector-element})$.
\end{lem}

\begin{lem}\label{lem:de-Rham-cocycle}
For $\lambda_0, \lambda_\tau\in E$ $(\tau\in \mathrm{Emb}(K,E) )$,
the extension of $E$ $($as a trivial $G_K$-module$)$ by $E$
corresponding to the cocycle $\lambda_0\psi_0+\sum_{\tau\in
\mathrm{Emb}(K,E)}\lambda_\tau\psi_\tau$ is de Rham if and only if
$\lambda_\tau=0$ for each $\tau$.
\end{lem}
\begin{proof} By \cite[Lemma 4.3]{Naka}, the subspace of extensions
of $E$ by $E$ that are de Rham is $1$-dimensional, and so consists
of those corresponding to the cocycles $\lambda_0\psi_0$
($\lambda_0\in E$).
\end{proof}

\subsection{An auxiliary formula}

Let $\vec{\CL}=(\CL_\sigma)_{\sigma: K\hookrightarrow E}$ be a
vector. We consider $\vec{\CL}$ as an element of $K\otimes_{\BQ_p}E$
via the isomorphism (\ref{eq:vector-element}).

Let $D$ be a filtered $E$-$(\varphi,N)$-module: the underlying
$E$-$(\varphi,N)$-module $D$ is a $(K_0\otimes_{\BQ_p}E)$-module
with a basis $\{f_1, f_2, f_3\}$ such that
$$\varphi^{[K_0:\BQ_p]} f_1 = p^{-[K_0:\BQ_p]} f_1, \
\varphi^{[K_0:\BQ_p]} f_2 =  f_2, \ \varphi^{[K_0:\BQ_p]} f_3 = f_3,
$$ and $$ N(f_1)=0, \ N(f_2)= - f_1, \ N(f_3)=f_1 ;  $$ the
filtration on $$ K\otimes_{K_0}D=(K\otimes_{\BQ_p}E) f_1 \oplus
(K\otimes_{\BQ_p}E) f_2 \oplus (K\otimes_{\BQ_p}E) f_3 $$ satisfies
$$ \Fil^iD = \left\{\begin{array}{cl} (K\otimes_{\BQ_p}E)(f_2- \vec{\CL}f_1) \oplus (K\otimes_{\BQ_p}E)(f_3 +
\vec{\CL}f_1) & \text{ if } i=0, \\ 0 & \text{ if } i>0 .\end{array}
\right.
$$

Let $\pi_i$ be the projection map
$$ \BX_\log(D)
\rightarrow \BB_{\log,E}, \ \ \ \sum_{j=1}^3 a_j f_j \mapsto a_i. $$

\begin{lem}\label{lem:aux} Let $c: G_K\rightarrow \BX_\log(D)$ be a
$1$-cocycle whose class in $H^1(G_K,\BX_\log(D))$ belongs to
$\ker(H^1(G_K,\BX_\log(D))\rightarrow H^1(G_K,\BX_\dR(D)))$. Then
there exist $$\gamma_{2,0}, \gamma_{2,\tau}, \gamma_{3,0},
\gamma_{3,\tau}\in E$$ $(\tau\in \mathrm{Emb}(K,E))$ such that
$$ \pi_2(c)=\gamma_{2,0}\psi_0+\sum_{\tau\in
\mathrm{Emb}(K,E)}\gamma_{2,\tau}\psi_\tau $$ and
$$ \pi_3(c)=\gamma_{3,0}\psi_0+\sum_{\tau\in
\mathrm{Emb}(K,E)}\gamma_{3,\tau}\psi_\tau. $$ Furthermore,
$$ \gamma_{2,0}-\gamma_{3,0}=\sum_{\tau\in
\mathrm{Emb}(K,E)}\CL_\tau(\gamma_{2,\tau}-\gamma_{3,\tau}). $$
\end{lem}

In our proof of Lemma \ref{lem:aux} we need the following

\begin{lem} \label{lem:same-kernel}
Let $D$ be an $E$-$(\varphi,N)$-module. If $\Fil_1$ and $\Fil_2$ are
two filtrations on $K\otimes_{K_0}D$ such that
$\Fil_1^0(K\otimes_{K_0}D)=\Fil^0_2(K\otimes_{K_0}D)$, then the
kernel of $$H^1(G_K, \BX_\log(D))\rightarrow H^1(G_K, \BX_\dR(D,
\Fil_1))$$ coincides with the kernel of $$H^1(G_K,
\BX_\log(D))\rightarrow H^1(G_K, \BX_\dR(D, \Fil_2)).$$
\end{lem}
\begin{proof} The proof is similar to that of \cite[Proposition 2.5]{Xie2015}
\end{proof}

\noindent{\it Proof of Lemma \ref{lem:aux}.} The argument is similar
to the proof of \cite[Lemma 5.1]{Xie2015}. We only give a sketch.

Write
$c_\sigma=\lambda_{1,\sigma}f_1+\lambda_{2,\sigma}f_2+\lambda_{3,\sigma}f_3$.
As $c$ takes values in $\BX_\log(D)$, we have $\lambda_{2,\sigma},
\lambda_{3,\sigma}\in E$. This ensures the existence of
$\gamma_{2,0}, \gamma_{2,\tau}, \gamma_{3,0},\gamma_{3,\tau}$.

Let $Fil$ be the filtration on $D$ such that $Fil^{-1}D=D$ and
$Fil^iD=\Fil^iD$ if $i\geq 0$. Then $(D, Fil)$ is admissible. Let
$V$ be the semistable $E$-representation of $G_K$ attached to
$D_V=(D, Fil)$. By Lemma \ref{lem:same-kernel}, $[c]$ is in the
kernel of $ H^1(G_K,\BX_\log(D_V))\rightarrow H^1(G_K,\BX_\dR(D_V))$
and so there exists a $1$-cocycle $c^{(1)}: G_K\rightarrow V$ such
that the image of $[c^{(1)}]$ by $H^1(G_K,V)\rightarrow
H^1(G_K,\BX_\log(D_V))$ is $[c]$.

We form the following commutative diagram
\begin{equation}\label{eq:comm-diag-galois-repn}
\xymatrix{ && V' \ar[d]  \\ 0 \ar[r] & V_0 \ar[r] \ar@{=}[d] & V
\ar[r]\ar[d]^{\pi_{V,V_1}} & T \ar[r]\ar[d] & 0
\\  0 \ar[r] & V_0 \ar[r] & V_1 \ar[r] & T_1 \ar[r] &
0 }\end{equation} with the horizontal lines being exact, where $V_0$
(resp. $V'$) is the subrepresentation of $V$ corresponding to the
filtered $E$-$(\varphi,N)$-submodule of $D_V$ generated by $f_1$
(resp. by $f_2 + f_3$) which is admissible. From
(\ref{eq:comm-diag-galois-repn}) we obtain the following commutative
diagram
$$ \xymatrix{ H^1(G_K, V) \ar[r]\ar[d]^{\pi_{V, V_1}} & H^1(G_K, T) \ar[r]\ar[d] & H^2(G_K, V_0)\ar@{=}[d] \\
H^1(G_K, V_1)\ar[r] & H^1(G_K, T_1) \ar[r] & H^2(G_K, V_0)  , } $$
where the horizontal lines are exact.

Write $c^{(2)}$ for the $1$-cocycle $G_K\xrightarrow{c^{(1)}}
V\rightarrow T\rightarrow T_1.$ By a simple computation we obtain
$$[c^{(2)}]=  [\Big( (\gamma_{2,0}
-\gamma_{3,0})\psi_0 +
\sum_{\tau\in\mathrm{Emb}(K,E)}(\gamma_{2,\tau}-\gamma_{3,\tau})\psi_\tau
\Big)\bar{f}_2],
$$ where $\bar{f}_2$ is the image of $f_2\in V$ in $T_1$.
Note that $T_1$ is isomorphic to $E$, and $V_0$ is isomorphic to
$E(1)$. Being the image of $[\pi_{V,V_1}(c^{(1)})]$ in $H^1(T_1)$,
$[c^{(2)}]$ lies in the kernel of $H^1(G_K,T_1)\rightarrow
H^2(G_K,V_0)$. By \cite[Lemma 5.5]{Zhang}, as an extension of $E$ by
$E(1)$, $V_1$ corresponds to the element $[(p)] + \exp(\vec{\CL})$.
Now Lemma \ref{lem:zhang} yields our second assertion. \qed

\section{$L$-invariants} \label{sec:L-inv}

Let $D$ be a filtered $E$-$(\varphi,N)$-module of rank $n$. Fix a
refinement $\CF$ of $D$. Then $\CF$ fixes an ordering $\alpha_1,
\cdots, \alpha_n$ of the eigenvalues of $\varphi^{[K_0:\BQ_p]}$ and
an ordering $\vec{k}_1, \cdots, \vec{k}_n$ of the Hodge-Tate
weights.

\subsection{The operator $N_{\CF}$}

The operator $\varphi$ induces a $K_0\otimes_{\BQ_p}E$-semilinear
operator $\varphi_\CF$ on $\gr^{\CF}_\bullet
D=\bigoplus\limits_{i=1}^n \CF_{i}D/\CF_{i-1}D$.

We define a $K_0\otimes_{\BQ_p}E$-linear operator $N_\CF$ on
$\gr^\CF_\bullet D$. The definition is similar to the one defined in
\cite{Xie2015}, so we omit some details.

For any $i\in \{1, \cdots, n\}$, if $N(\CF_i D)=N (\CF_{i-1}D)$, we
demand that $N_\CF$ maps $\gr^\CF_i D$ to zero.

Now we assume that $N(\CF_i D) \supsetneq N (\CF_{i-1} D)$. Let $j$
be the minimal integer such that $$ N (\CF_i D) \subseteq
N(\CF_{i-1}D) + \CF_j D.
$$

\begin{prop}
$N (\CF_{i-1}D) \cap \CF_j D = N (\CF_{i-1}D)\cap \CF_{j-1}D$.
\end{prop}
\begin{proof} Note that $\CF_jD$, $\CF_{j-1}D$, $N (\CF_{i-1}D) + \CF_j D$ and
$N(\CF_{i-1}D)+\CF_{j-1}D$ are stable by $\varphi$. Thus $(N
(\CF_{i-1}D) + \CF_j D)/(N (\CF_{i-1}D) + \CF_{j-1} D)$ is a
$\varphi$-module, and so must be free over $K_0\otimes_{\BQ_p}E$.
Hence the map
\begin{equation} \label{eq:N-CF-iso}
\CF_jD/\CF_{j-1}D \rightarrow  (N (\CF_{i-1}D) + \CF_j D)/(N
(\CF_{i-1}D) + \CF_{j-1} D) \end{equation} is an isomorphism. It
follows that $N (\CF_{i-1}D) \cap \CF_j D = N (\CF_{i-1}D)\cap
\CF_{j-1}D$.
\end{proof}

The operator $N$ induces a $K_0\otimes_{\BQ_p}E$-linear map
$$ \CF_iD/\CF_{i-1}D \rightarrow (N
(\CF_{i-1}D) + \CF_j D)/(N (\CF_{i-1}D) + \CF_{j-1} D). $$ We define
the map $N_\CF: \gr_i^\CF D \rightarrow \gr^\CF_{j}D$ to be the
composition of this map and the inverse of (\ref{eq:N-CF-iso}).

Finally we extend $N_\CF$ to the whole $\gr^\CF_\bullet D$ by
$K_0\otimes_{\BQ_p}E$-linearity. Note that
$N_\CF\varphi_\CF=p\varphi_\CF N_\CF$. By definition, for any $i$ we
have either $N(\gr^\CF_i D) = 0$ or $N(\gr^\CF_i D)= \gr^\CF_j D$
for some $j$.

\begin{defn} For $j\in \{1,\cdots, n-1\}$ we say that $j$ is {\it marked} (or a {\it marked index}) for $\CF$
if there is some $i\in \{2, \cdots, n\}$ such that $N_\CF(\gr^\CF_i
D)= \gr^\CF_j D$.
\end{defn}

Note that $i$ and $j$ in the above definition are determined by each
other. We write $i=t_\CF(j)$ and $j=s_\CF(i)$.

\begin{prop}  The following two assertions are
equivalent:
\begin{enumerate}
\item\label{it:critical-a} $s$ is marked and $t=t_\CF(s)$.
\item\label{it:critical-b} $N\CF_{t-1}D \cap \CF_s D = N\CF_{t-1}D \cap \CF_{s-1} D$ and
$N\CF_{t}D \cap \CF_s D \supsetneq N\CF_{t}D \cap \CF_{s-1} D$.
\end{enumerate}
\end{prop}
\begin{proof} We have already seen that, if (\ref{it:critical-a}) holds, then (\ref{it:critical-b}) holds.
Conversely, we assume that (\ref{it:critical-b}) holds. Then
$N\CF_{t}D \cap \CF_s D \supsetneq N \CF_{t-1}D\cap \CF_sD$. Thus
$N\CF_{t}D \supsetneq N \CF_{t-1}D$.

We show that $N\CF_tD \varsubsetneq N \CF_{t-1}D +\CF_{s-1}D$. If it
is not true, then there exists $y\in \CF_tD \backslash \CF_{t-1}D$
which is a lifting of a basis of $\gr^\CF_tD$ over
$K_0\otimes_{\BQ_p}E$ such that $N(y)\in \CF_{s-1}D$. For any $z\in
\CF_tD$, write $z=w+\lambda y$ with $w\in \CF_{t-1}D$ and
$\lambda\in K_0\otimes_{\BQ_p}E$. If $N(z)$ is in $\CF_s D$, then
$N(w)$ is also in $\CF_s D$. But $N\CF_{t-1}D\cap \CF_s D
=N\CF_{t-1}D\cap \CF_{s-1}D$. Thus $N(w)$ is in $\CF_{s-1}D$, which
implies that $N(z)=N(w)+\lambda N(y)$ is also in $\CF_{s-1}D$. So,
$N\CF_tD\cap \CF_s D=N\CF_tD\cap \CF_{s-1}D$, a contradiction.

From $N\CF_{t}D \cap \CF_s D \supsetneq N \CF_{t-1}D\cap \CF_sD$ we
see that there is $x\in \CF_tD \backslash \CF_{t-1}D$ such that
$N(x)\in \CF_sD$. We must have  $N \CF_t D \subseteq N \CF_{t-1}D +
\CF_sD$. Otherwise, let $j$ be the smallest integer such that $N
\CF_t D \subseteq N \CF_{t-1}D + \CF_jD$ and assume that $j>s$. Then
$N_\CF(x+\CF_{t-1}D)=0$, which contradicts the fact that $N_\CF:
\gr^\CF_t D \rightarrow \gr^\CF_j D$ is an isomorphism.
\end{proof}

\subsection{Strongly marked indices and
$\CL$-invariants} \label{ss:strong-crit-L-inv}

Assume that $s$ is marked for $\CF$ and $t=t_\CF(s)$. We consider
the decompositions
$$ \CF_t D / \CF_{s-1} D = (K_0\otimes_{\BQ_p} E)\cdot \bar{e}_s \oplus L \oplus (K_0\otimes_{\BQ_p} E) \bar{e}_t
$$
that satisfy the following conditions:

$\bullet$ $\overline{\CF}_1 (\CF_t D / \CF_{s-1}
D)=(K_0\otimes_{\BQ_p}E)\bar{e}_s$ and
$\overline{\CF}_{t-s}(\CF_tD/\CF_{s-1} D) = (K_0\otimes_{\BQ_p}E)
\bar{e}_s\oplus L$, where $\overline{\CF}$ is the refinement on
$\CF_tD/\CF_{s-1} D$ induced by $\CF$.

$\bullet$ Both $L$ and $(K_0\otimes_{\BQ_p}E) \bar{e}_s \oplus
(K_0\otimes_{\BQ_p}E) \bar{e}_t$ are stable by $\varphi$ and $N$;
$\varphi^{[K_0:\BQ_p]}(\bar{e}_t)=\alpha_t \bar{e}_t$ and
$N(\bar{e}_t)=\bar{e}_s$.

\noindent Such a decomposition is called an {\it $s$-decomposition.}

\begin{rem} $s$-decompositions may be not exist. However, if $\varphi$ is
semisimple, then $s$-decompositions always exist (see
\cite{Xie2015}).
\end{rem}

Let $\mathrm{dec}$ denote an $s$-decomposition $ \CF_t D / \CF_{s-1}
D = E \bar{e}_s \oplus L \oplus E \bar{e}_t. $

There is a natural isomorphism $E \bar{e}_s \oplus E
\bar{e}_t\rightarrow (\CF_tD/\CF_{s-1}D)/L$ of
$(\varphi,N)$-modules. Usually the filtration on the filtered
$E$-$(\varphi,N)$-submodule $E \bar{e}_s \oplus E \bar{e}_t$ and
that on $(\CF_tD/\CF_{s-1}D)/L$ are different.

When these two filtrations satisfy certain compatible condition, we
say the decomposition $\mathrm{dec}$ is perfect. Precisely, we say
that $\mathrm{dec}$ is {\it perfect} if for any $\tau:
K\hookrightarrow E$ we have $k_{s,\tau}< k_{t, \tau}$, and if there
exist $k'_{s,\tau}, k'_{t,\tau}$ and $\CL_{\mathrm{dec},\tau}\in E$
satisfying $k_{s,\tau}\leq k'_{s,\tau}< k'_{t,\tau}\leq k_{t,\tau}$
such that the following conditions hold.

$\bullet$ The filtration on the filtered $E$-$(\varphi,N)$-submodule
$E \bar{e}_s \oplus E \bar{e}_t$ satisfies
$$\Fil^{i}_\tau (E \bar{e}_s \oplus E \bar{e}_t) = \left\{ \begin{array}{ll} E \bar{e}_{s,\tau}\oplus E \bar{e}_{t,\tau} & \text{ if }i \leq k_{s,\tau}, \\
E(\bar{e}_{t,\tau} + \CL_{\mathrm{dec,\tau}} \bar{e}_{s,\tau}) & \text{ if } k_{s,\tau} < i\leq k'_{t,\tau} ,\\
0 & \text{ if } i > k'_{t,\tau},
\end{array}\right.$$

$\bullet$ The filtration on the quotient of $\CF_tD/\CF_{s-1}D$ by
$L$ satisfies
$$\Fil^{i}_\tau \CF_tD/\CF_{s-1}D = \left\{ \begin{array}{ll} E \bar{e}_{s,\tau}\oplus E \bar{e}_{t,\tau} & \text{ if }i \leq k'_{s,\tau}, \\
E(\bar{e}_t + \CL_{\mathrm{dec},\tau} \bar{e}_s) & \text{ if } k'_{s,\tau} < i\leq k_{t,\tau} ,\\
0 & \text{ if } i > k_{t,\tau},
\end{array}\right.$$ where the images of $\bar{e}_{s}$ and
$\bar{e}_t$ in $\CF_tD/\CF_{s-1}D$ are again denoted by $\bar{e}_s$
and $\bar{e}_t$.

\begin{defn}\label{defn:Fontaine-Mazur}
If there exists a perfect $s$-decomposition, we say that $s$ is {\it
strongly marked} (or a {\it strongly marked index}). In this case we
attached to each pair $(s,t)$  with $t=t_\CF(s)$ an invariant
$\vec{\CL}_{\CF,s,t}=(\CL_{\mathrm{dec},\tau})_{\tau}$, where
$\mathrm{dec}$ is a perfect $s$-decomposition. Proposition
\ref{prop:fon-mazur-inv} below tells us that $\vec{\CL}_{\CF,s,t}$
is independent of the choice of perfect $s$-decompositions. We call
$\vec{\CL}_{\CF,s,t}$ the {\it Fontaine-Mazur $\CL$-invariant}
associated to $(\CF,s,t)$, and denote $\CL_{\mathrm{dec},\tau}$ by
$\CL_{\CF, s,t, \tau}$.
\end{defn}

In the case of $t=s+1$, $s$ is strongly marked if and only if
$k_{s,\tau}<k_{t,\tau}$ for all $\tau$.

\begin{prop} \label{prop:fon-mazur-inv}
If $\mathrm{dec}_1$ and $\mathrm{dec}_2$ are two perfect
$s$-decompositions, then $\CL_{\mathrm{dec}_1,
\tau}=\CL_{\mathrm{dec}_2,\tau}$ for any $\tau$.
\end{prop}
\begin{proof} The argument is similar to the proof of \cite[Proposition
4.9]{Xie2015}.
\end{proof}

Let $D^*$ be the filtered $E$-$(\varphi, N)$-module that is the dual
of $D$. Let $\check{\CF}$ be the refinement on $D^*$ such that
$$ \check{\CF}_i D^* : = (\CF_{n-i}D)^\bot=\left\{y\in D^*: \langle y, x
\rangle =0 \text{ for all }x\in \CF_{n-i}D \right\}. $$ We call
$\check{\CF}$ the {\it dual refinement} of $\CF$.

If $L\subset M$ are submodules of $D$, then $M^\bot\subset L^\bot$.
The pairing $\langle \cdot, \cdot \rangle:  L^\bot\times M$ induces
a non-degenerate pairing on $L^\bot/M^\bot \times M/L  $, so that we
can identify $L^\bot/M^\bot$ with the dual of $M/L$ naturally. In
particular, $\gr^{\check{\CF}}_i D^*$ is naturally isomorphic to the
dual of $\gr^{\CF}_{n+1-i}D$. Thus $\gr^{\check{\CF}}_\bullet D^*$
is naturally isomorphic to the dual of $\gr^\CF_\bullet D$.

\begin{prop}\label{prop:duality-critical}
\begin{enumerate}
\item \label{it:duality-critical-a}
$N_{\check{\CF}}$ is dual to $-N_\CF$.
\item \label{it:duality-critical-b}
$s$ is marked for $\CF$ if and only if $n+1-t_\CF(s)$ is marked for
$\check{\CF}$.
\item \label{it:duality-critical-c}
$s$ is strongly marked for $\CF$ if and only if $n+1-t_\CF(s)$ is
strongly marked for $\check{\CF}$.
\end{enumerate}
\end{prop}
\begin{proof} The proof of (\ref{it:duality-critical-a}) is similar to that of \cite[Proposition 4.14]{Xie2015}. The proof of
(\ref{it:duality-critical-b}) is similar to that of
\cite[Proposition 4.13]{Xie2015}. The proof of
(\ref{it:duality-critical-c}) is similar to that of
\cite[Proposition 4.15 (a)]{Xie2015}.
\end{proof}

\section{Projection vanishing property} \label{sec:proj-van-prop}

Put $S=E[Z]/(Z^2)$. Let $z$ be the closed point defined by the
maximal ideal $(Z)$ of $S$.

Let $W=(W_e, W^+_\dR)$ be an $S$-$B$-pair. Let $\{w_1, \cdots,
w_n\}$ be a $\BB_{e,S}$-basis of $W_e$. Suppose that $W$ admits a
triangulation $\Fil_\bullet$. Let $(\delta_{1}, \cdots, \delta_{n})$
be the corresponding triangulation parameters. Then for each
$i=1,\cdots, n$ there exists a continuous additive character
$\epsilon_i$ of $K^\times$ with values in $E$ such that $\delta_i =
\delta_{i,z} (1+Z \epsilon_i)$.

Suppose that $W_z$, the evaluation of $W$ at $z$, is semistable, and
let $D_z$ be the filtered $E$-$(\varphi,N)$-module attached to
$W_z$. Let $\CF$ be the refinement of $D_z$ corresponding to the
induced triangulation of $W_z$, and let $\{e_{1,z}, e_{2,z}, \cdots,
e_{n,z}\}$ be a $(K_0\otimes_{\BQ_p}E)$-basis of $D_z$ that is
compatible with $\CF$ i.e. $\CF_i D= (K_0\otimes_{\BQ_p}E)e_{1,z}
\oplus \cdots \oplus (K_0\otimes_{\BQ_p}E) e_{i,z}$. Let
$\alpha_{i,z}\in E$ be such that
$\varphi^{[K_0:\BQ_p]}(e_{i,z})=\alpha_{i,z}e_{i,z} \mod \CF_{i-1}$.

Let $x_{ij}\in \BB_{\log,E}$ ($i,j=1,\cdots, n$) be such that
\begin{equation}\label{eq:relation-e-and-v}
e_{i,z} = x_{1i} w_{1,z} +\cdots + x_{ni} w_{n,z} .   \end{equation}
Then $X=(x_{ij})$ is in $\GL_n(\BB_{\log,E})$. Write the matrix of
$\sigma\in G_K$ with respect to the basis $\{w_{1}, \cdots, w_{n}\}$
by $(I_n + Z U_{e,\sigma})A_{e,\sigma}$. As $e_{1, z},\cdots,
e_{n,z}$ are fixed by $G_{K}$, we have $X^{-1}A_{e,\sigma}
\sigma(X)=I_n$ for all $\sigma\in G_{K}$.

For $i=1,\cdots, n$ put $e_i = x_{1i} w_1 +\cdots + x_{ni}  w_n$.
Then $\{e_1,\cdots, e_n\}$ is a basis of $\BB_{\log,S}\otimes_S W_e$
over $\BB_{\log,S}$.

\begin{lem} \label{lem:trivial-useful} If $T$ is the matrix of $\varphi_{D_z}$ for the basis $\{e_{1,z}, \cdots,
e_{n,z}\}$, then $T$ is also the matrix of
$\varphi_{\BB_{\log,S}\otimes_S W_e}$ for the basis $\{e_{1},
\cdots, e_{n}\}$.
\end{lem}
\begin{proof} The assertion follows from the definition of $\{e_1, \cdots, e_n\}$ and the fact that $w_{1, z}, \cdots, w_{n,z}, w_1, \cdots,
w_n$ are fixed by $\varphi$.
\end{proof}

In Section \ref{sec:coh-B-pair} we attach to $W$ an element $c_B(W)$
in $H^1_B(W_z^*\otimes W_z)$. Consider the composition
$$ H^1_B(W_z^*\otimes W_z) \rightarrow H^1(G_K, W_{e,z}^*\otimes_{\BB_{e,E}} W_{e,z}) \rightarrow H^1(G_K,
\BB_{\log,E}\otimes_E (D_z^*\otimes D_z)) . $$ As the matrix of
$\sigma\in G_K$ for the basis $\{e_1,\cdots, e_n\}$ is $I_n + Z
X^{-1}U_{e,\sigma} X$, from the discussion in Section \ref{sec:coh}
we see that the image of $c_B$ in $H^1(G_K, \BB_{\log,E}\otimes_E
(D_z^*\otimes D_z))$ is the class of the $1$-cocycle
$$(U_{e,\sigma})_{ij} w^*_{j,z}\otimes w_{i,z}  =(X^{-1}U_{e,\sigma}X)_{ij} e^*_{j,z}\otimes e_{i,z} . $$
Let $\pi_{h \ell}$ be the projection
\begin{equation}\label{eq:proj}
\BB_{\log,E}\otimes_E (D_z^*\otimes D_z)\rightarrow \BB_{\log,E} , \
\ \ \sum_{j,i} b_{ji} e^*_{j,z}\otimes e_{i,z} \mapsto b_{h \ell} .
\end{equation}

For $h=1, \cdots, n$, let $\epsilon'_h$ be the additive character of
$G_K$ such that $\epsilon'_h \circ\mathrm{rec}_K(p)=0$ and
$\epsilon'_h\circ\mathrm{rec}_K|_{\fo_K^\times}=\epsilon_h|_{\fo_K^\times}$.

\begin{thm} \label{prop:middle-step}
\begin{enumerate}
\item\label{it:middle-step-a} For any pair of integers $(h,\ell)$ such that $h<\ell$ we have
$\pi_{h \ell}([c])=0$.
\item\label{it:middle-step-b} For any $h=1,\cdots, n$,
$\pi_{h,h}([c])$ coincides with the image of $[\epsilon'_h]$ in
$H^1(G_K, \BB_{\log,E})$.
\end{enumerate}
\end{thm} We call (\ref{it:middle-step-a}) the {\it projection
vanishing property}.
\begin{proof} The filtered $E$-$(\varphi,N)$-module attached to $W_z/
\Fil_{h-1}W_z$ is $D_z/ \CF_{h-1}D_z$. We denote the image of
$e_{\ell,z}$ ($\ell\geq h$) in $D_z/\CF_{h-1}D_z$ again by
$e_{\ell,z}$.

Let $\delta'_h$ be the character of $G_K$ such that
$\delta'_h=1+Z\epsilon'_h$. By Lemma \ref{lem:varph-gamma-fil} there
exists an element
$$x\in (\BB_{\mathrm{max},E}\otimes_{ \BB_{e,E} } (W/
\Fil_{h-1}W)_e)^{G_K=\delta'_h,
\varphi^{[K_0:\BQ_p]}=\alpha_{i,z}(1+Zv_p(\pi_K)\epsilon_h(p))}$$
whose image in $D_z/\CF_{h-1} D_z$ is $e_{h,z}$.  Write $x= e_{h}+
Z\sum\limits_{\ell\geq h} \lambda_\ell e_{\ell}$ with
$\lambda_\ell\in \BB_{\log,E}$.

As the matrix of $\sigma\in G_K$ for the basis $\{e_{1}, \cdots,
e_n\}$ is $I_n+ZX^{-1}U_{e,\sigma}X$, we have \begin{eqnarray*} &&
[1+Z\epsilon'_h(\sigma)]x =[1+Z\epsilon'_h(\sigma)]( e_{h}+ Z\sum_{\ell\geq h} \lambda_\ell e_{\ell}) \\
& = &\sigma(x)=   e_{h} + Z\sum_{\ell\geq h}
(X^{-1}U_{e,\sigma}X)_{\ell h}e_{\ell} + Z\sum_{\ell\geq
h}\sigma(\lambda_\ell)e_{\ell} .
\end{eqnarray*} For $\ell>h$,
comparing the coefficients of $e_{\ell}$ we obtain
$$(X^{-1}U_{e,\sigma}X)_{\ell h}= (1-\sigma)\lambda_\ell, $$ which shows
(\ref{it:middle-step-a}). Similarly, comparing coefficients of
$e_{h}$ we obtain
\begin{equation} \label{eq:Uhh-episilon}
(X^{-1}U_{e,\sigma}X)_{hh}-\epsilon'_h(\sigma) = (1-\sigma)\lambda_h
,\end{equation} which implies (\ref{it:middle-step-b}).
\end{proof}

\section{The proof of Theorem \ref{thm:main}} \label{sec:main}

We will need the following lemmas.

\begin{lem}\label{cor:inc-iso} The inclusion $ E \hookrightarrow \BB_{e,E} $ induces an
isomorphism
$$ H^1(G_K,E) \xrightarrow{\sim} \ker(N: H^1(G_K, \BB_{e,E})\rightarrow H^1(G_K, \BB_{\log, E})).
$$
\end{lem}
\begin{proof} The proof is identical to that of \cite[Corollary
1.4]{Xie2015}.
\end{proof}

\begin{lem} \label{lem:N-surj} The map
$N: \BB_{\log,E}^{\varphi=p}\rightarrow \BB_{\log, E}^{\varphi=1}$
is surjective.
\end{lem}
\begin{proof} The proof is identical to that of \cite[Lemma 1.2]{Xie2015}.
\end{proof}

For the proof of Theorem \ref{thm:main} we may assume that
$S=E[Z]/(Z^2)$, and $z$ is the closed point defined by the maximal
ideal $(Z)$. Let $W$ be as in Theorem \ref{thm:main}. Replacing $W$
by the $E$-$B$-pair $\CF_tW/\CF_{s-1}W$ and replacing $\CF$ by the
induced refinement on $\CF_tW/\CF_{s-1}W$, we may assume that $s=1$
and $t=n=\mathrm{rank}_{\BB_{e,E}}(W_e)$. Let $e_{1,z}, e_{2,z},
\cdots, e_{n,z}$ be a $K_0\otimes_{\BQ_p}E$-basis of $D_z$ such that
\begin{equation} \label{eq:s-decomp}
(K_0\otimes_{\BQ_p}E)e_{1,z} \bigoplus  L \bigoplus
(K_0\otimes_{\BQ_p}E) e_{n,z}
\end{equation} with $L=\oplus_{i=2}^{n-1}(K_0\otimes_{\BQ_p}E)e_{i,z}$
a perfect $1$-decomposition of $D_z$ for $\CF$ (see \S
\ref{ss:strong-crit-L-inv} for the meaning of perfect
decompositions). Let $e_{1,z}^*, e_{2,z}^*, \cdots, e^*_{n,z}$ be
the dual basis of $D^*_z$ over $K_0\otimes_{\BQ_p}E$.

Let $D_1$ be the quotient of $D_z$ by $L$, $D_2^*$ the quotient of
$D_z^*$ by $\oplus_{i=2}^{n-1}(K_0\otimes_{\BQ_p}E)e^*_{i,z}$. Put
$\mathscr{D}=D_2^*\otimes D_1$. The images of $e_{1,z}$ and
$e_{n,z}$ in $D_1$ are again denoted by $e_{1,z}$ and $e_{n,z}$, and
the images of $e^*_{1,z}$ and $e^*_{n,z}$ in $D_2^*$ are again
denoted by $e^*_{1,z}$ and $e^*_{n,z}$ respectively. So
$e^*_{1,z}\otimes e_{1,z}$, $e^*_{1,z}\otimes e_{n,z}$,
$e^*_{n,z}\otimes e_{1,z}$, $e^*_{n,z}\otimes e_{n,z}$ form a
$K_0\otimes_{\BQ_p}E$-basis of $\mathscr{D}$. Let $\mathscr{D}_0$ be
the filtered $E$-$(\varphi,N)$-submodule of $\mathscr{D}$ with a
$K_0\otimes_{\BQ_p}E$-basis $\{e^*_{1,z}\otimes e_{1,z}, \
e^*_{n,z}\otimes e_{1,z}, \ e^*_{n,z}\otimes e_{n,z}\}$. Let
$\mathscr{W}=(\mathscr{W}_e,\mathscr{W}^+_\dR)$ (resp.
$\mathscr{W}_0$) be the $E$-$B$-pair attached to $\mathscr{D}$
(resp. $\mathscr{D}_0$). Note that
$$\varphi^{[K_0:\BQ_p]}(e^*_{1,z}\otimes e_{1,z})=e^*_{1,z}\otimes e_{1,z}, \
\varphi^{[K_0:\BQ_p]}(e^*_{n,z}\otimes e_{n,z})=e^*_{n,z}\otimes
e_{n,z},
$$
$$ \varphi^{[K_0:\BQ_p]}(e^*_{n,z}\otimes e_{1,z})= p^{-[K_0:\BQ_p]}
e^*_{n,z}\otimes e_{1,z},
$$ and
$$ -N(e^*_{1,z}\otimes e_{1,z}) = N(e^*_{n,z}\otimes e_{n,z}) = e^*_{n,z}\otimes e_{1,z}, \ N(e^*_{n,z}\otimes e_{1,z}) = 0 $$

Let $\vec{\CL}_{\CF}=\vec{\CL}_{\CF,s,t}$ be the $\CL$-invariant
defined in Definition \ref{defn:Fontaine-Mazur}. As
(\ref{eq:s-decomp}) is a prefect decomposition, we have
\begin{eqnarray*}
\Fil^0 (K\otimes_{K_0}\mathscr{D})  = && E e^*_{n,z} \otimes
(e_{n,z} +\vec{\CL}_{\CF}e_{1,z}) \oplus E
(e^*_{1,z}-\vec{\CL}_{\CF}e^*_{n,z})\otimes e_{1,z} \\ && \oplus E
(e^*_{1,z}-\vec{\CL}_{\CF}e^*_{n,z})\otimes (e_{n,z}
+\vec{\CL}_{\CF}e_{1,z}). \end{eqnarray*} and \[ \Fil^0
(K\otimes_{K_0}\mathscr{D}_0) = E e^*_{n,z} \otimes (e_{n,z}
+\vec{\CL}_{\CF}e_{1,z}) \oplus E
(e^*_{1,z}-\vec{\CL}_{\CF}e^*_{n,z})\otimes e_{1,z} .
\]

Consider $W$ as an infinitesimal deformation of $W_z$. In Section
\ref{ss:cocycle-inform} we attach to this infinitesimal deformation
an element $c_B(W)$ in $H^1_B(W_z^*\otimes W_z)$. Let $[c]$ be the
image of $c_B(W)$ by the composition
$$ H^1_B(W_z^*\otimes W_z)\rightarrow H^1(G_K, W_{e,z}^*\otimes_{\BB_{e,E}}W_{e,z})\rightarrow H^1(G_K,
\BB_{\log,E}\otimes_{K_0\otimes_{\BQ_p}E} (D_z^*\otimes D_z)), $$
and choose a $1$-cocyle $c$ representing $[c]$. Write $c$ in the
form
$$ c =  \sum_{j,i} c_{j,i} e^*_{j,z}\otimes e_{i,z}
$$ with $c_{i,j}$ being a $1$-cocycle of $G_K$ with values in $\BB_{\log,E}$. By the projection vanishing property (Theorem \ref{prop:middle-step}
(\ref{it:middle-step-a})) we have $[c_{1,n}]=0$.

\begin{lem}\label{lem:constants} There exist $\xi_1, \xi_n \in \BB_{e,E}$ and $\gamma_{1,0},
\gamma_{1,\tau}, \gamma_{n,0}, \gamma_{n,\tau}$ $(\tau\in
\mathrm{Emb}(K,E))$ such that
$$ c_{1,1}(\sigma) =(\sigma-1)\xi_1 +  \gamma_{1,0} \psi_0(\sigma) +\sum_{\tau\in \mathrm{Emb}(K,E)} \gamma_{1,\tau} \psi_\tau (\sigma)  $$
and
$$ c_{n,n}(\sigma) =(\sigma-1)\xi_n +  \gamma_{n,0} \psi_0(\sigma) +\sum_{\tau\in \mathrm{Emb}(K,E)} \gamma_{n,\tau} \psi_\tau (\sigma)
$$ for any $\sigma\in G_K$. \end{lem}
\begin{proof} Let $\bar{c}_B$ be the image of $c_B$ in $H^1_B(\mathscr{W})$, and let $\bar{c}$ be the $1$-cocycle
$$ \bar{c} =  \sum_{j,i\in \{1,n\}} c_{j,i} e^*_{j,z}\otimes e_{i,z}
$$ of $G_K$ with values in $\BB_{\log,E}\otimes_{K_0\otimes_{\BQ_p}E}
\mathscr{D}$. Then the image of $\bar{c}_B$ in
$$H^1(G_K,\BB_{\log,E}\otimes_{K_0\otimes_{\BQ_p}E} \mathscr{D} )$$ is
$[\bar{c}]$.

Note that $\bar{c}$ has values in
$\mathscr{W}_e=(\BB_{\log,E}\otimes_{K_0\otimes_{\BQ_p}E}
\mathscr{D})^{\varphi=1,N=0}$. So, in particular $c_{1,1}$ and
$c_{n,n}$ have values in $\BB_{e,E}$.
%
%
As $N\bar{c}=0$, we have
\begin{eqnarray*}  N(c_{n,1})= c_{1,1} - c_{n,n}, \ \ \
 -N(c_{1,1}) = N(c_{n,n})= c_{1,n}.
\end{eqnarray*}
As $[c_{1,n}]=0$, the statement follows from Lemma
\ref{cor:inc-iso}.
\end{proof}

Write $ \delta_i=\delta_{i,z}(1+Z\epsilon_i) $. Let $\epsilon'_i$ be
the additive character of $G_K$ with values in $E$ such that
$\epsilon'_i\circ \mathrm{rec}_K(p)=0$ and $\epsilon'_i\circ
\mathrm{rec}_K|_{\fo_K^\times}=\epsilon_i|_{\fo_K^\times}$. Then
there are $\epsilon_{i,\tau}$ ($\tau\in \mathrm{Emb}(K, E)$) such
that $\epsilon'_i=\sum\limits_{\tau\in\mathrm{Emb}(K,E)}
\epsilon_{i,\tau} \psi_{\tau}$.

\begin{lem} \label{lem:constant-gamma}
For $h=1,n$ we have
$[K_0:\BQ_p]\gamma_{h,0}=-v_p(\pi_K)\epsilon_h(p)$ and
$\gamma_{h,\tau}=\epsilon_{h,\tau}$.
\end{lem}
\begin{proof}
We keep to use notations in the proof of Theorem
\ref{prop:middle-step}. By (\ref{eq:Uhh-episilon}) and Lemma
\ref{lem:constants} we have
\begin{eqnarray*} (\sigma-1)(\lambda_h ) & = &-(X^{-1} U_\sigma
X)_{hh}+\sum_{\tau\in \mathrm{Emb}(K,E)}
\epsilon_{h,\tau}\psi_\tau(\sigma) \\
&=&
-(\sigma-1)\xi_h-\gamma_{h,0}\psi_0(\sigma)+\sum_{\tau\in\mathrm{Emb}(K,E)}(\epsilon_{h,\tau}-\gamma_{h,\tau})\psi_\tau(\sigma).
\end{eqnarray*} Note that there exists $\omega\in
\mathrm{W}(\overline{\BF}_p)$ such that $\varphi(\omega)-\omega=1$,
where $\mathrm{W}(\overline{\BF}_p)$ is the ring of Witt vectors
with coefficients in the algebraic closure of $\BF_p$. Then
$(\sigma-1)\omega= \psi_0(\sigma) $. Hence
$$ \sum_{\tau\in\mathrm{Emb}(K,E)}(\epsilon_{h,\tau}-\gamma_{h,\tau})\psi_\tau(\sigma)
= (\sigma-1)(\lambda_h +\xi_h +\gamma_{h,0} \omega).
$$ In other words, the cocycle
$\sum\limits_{\tau\in\mathrm{Emb}(K,E)}(\epsilon_{h,\tau}-\gamma_{h,\tau})\psi_\tau(\sigma)$
is de Rham. By Lemma \ref{lem:de-Rham-cocycle} we have
$\gamma_{h,\tau} = \epsilon_{h,\tau}$ and $\lambda_h+\xi_h
+\gamma_{h,0} \omega \in E$. Then \begin{equation}
\label{eq:lambda-a} (\varphi^{[K_0:\BQ_p]}-1)\lambda_h =
-(\varphi-1)\xi_h - \gamma_{h,0} (\varphi^{[K_0:\BQ_p]}-1)\omega =
-[K_0:\BQ_p]\gamma_{h,0} .
\end{equation}

By our choice of the basis $\{e_{1,z}, \cdots, e_{n,z}\}$,  $Y_1=
\oplus_{i=2}^n Z e_{i,z} $ is stable by $\varphi$. Put $Y_n=0$. Let
$x$ be as in the proof of Theorem \ref{prop:middle-step}. By Lemma
\ref{lem:trivial-useful} we have $\varphi^{[K_0:\BQ_p]} e_{h,z}
=\alpha_{h,z}e_{h,z}$. Thus for $h=1, n$ we have
$$ \varphi^{[K_0:\BQ_p]}(x) =  ( 1+Z \varphi^{[K_0:\BQ_p]}(\lambda_h) )\alpha_{h,z} e_h  \hskip 10pt (\text{mod } Y_h).
$$ On the other hand,
\begin{eqnarray*} \varphi^{[K_0:\BQ_p]}(x) & = & (1+Z v_p(\pi_K)\epsilon_h(p)) \alpha_{h,z} x \\
& = & (1+Z  v_p(\pi_K)\epsilon_h(p))   \alpha_{h,z} (1+ Z
\lambda_{h}) e_h \hskip 10pt (\text{mod } Y_h).
\end{eqnarray*} Hence we obtain
\begin{equation} \label{eq:lambda-h}
(\varphi^{[K_0:\BQ_p]}-1)\lambda_h =
v_p(\pi_K)\epsilon_h(p).\end{equation} By (\ref{eq:lambda-a}) and
(\ref{eq:lambda-h}) we have
$$ [K_0:\BQ_p]\gamma_{h,0} = -(\varphi^{[K_0:\BQ_p]}-1) \lambda_h = -v_p(\pi_K)\epsilon_h(p), $$ as
wanted.
\end{proof}

By Lemma \ref{lem:N-surj} there exists some $y\in
\BB_{\log,E}^{\varphi=p}$ such that $N(y)=\xi_1-\xi_n$. Let
$\bar{c}'$ be the $1$-cocycle of $G_K$ with values in
$\BB_{\log,E}\otimes_{K_0\otimes_{\BQ_p}E} \mathscr{D}_0$
 such that $$ \bar{c}'= c'_{1,1}
e^*_{1,z}\otimes e_{1,z} + c'_{n,n} e^*_{n,z}\otimes e_{n,z} +
c'_{n,1}e_{n,z}^*\otimes e_{1,z} $$ with
$$ c'_{1,1} = \gamma_{1,0} \psi_0  +\sum_{\tau\in \mathrm{Emb}(K,E)} \gamma_{1,\tau} \psi_\tau
, \ \ c'_{n,n} = \gamma_{n,0} \psi_0  +\sum_{\tau\in
\mathrm{Emb}(K,E)} \gamma_{n,\tau} \psi_\tau  $$ and
$$ c'_{n,1}(\sigma) = c_{n,1}(\sigma) - (\sigma-1)y , \ \ \sigma\in G_K.
$$ It is easy to check that $\varphi(\bar{c}')=\bar{c}'$ and
$N(\bar{c}')=0$. Hence $\bar{c}'$ is a $1$-cocycle of $G_K$ with
values in $\BX_\log(\mathscr{D}_0)$.

\begin{prop} The image of $[\bar{c}']$ in $H^1(G_K,
\BX_\log(\mathscr{D}_0))$ belongs to the kernel of $$ H^1(G_K,
\BX_\log(\mathscr{D}_0))\rightarrow H^1(G_K,
\BX_\dR(\mathscr{D}_0)).$$
\end{prop}
\begin{proof}
Consider the following commutative diagram
\[ \xymatrix{ H^1(G_K, \BX_\log(\mathscr{D}_0)) \ar[r] \ar[d] & H^1(G_K,
\BX_\dR(\mathscr{D}_0)) \ar[d] \\
H^1(G_K, \BX_\log(\mathscr{D})) \ar[r]  & H^1(G_K,
\BX_\dR(\mathscr{D})). } \] The right vertical arrow in the above
diagram is injective (see \cite[Corollary 2.4]{Xie2015}). So we only
need to show that the image of $[\bar{c}']$ in $H^1(G_K,
\BX_\dR(\mathscr{D}))$ is zero. Note that
$$[\bar{c}']=[\bar{c}]-[c_{1,n} e^*_{1,z}\otimes e_{n,z}]= -[c_{1,n}
e^*_{1,z}\otimes e_{n,z}]$$ in $H^1(G_K, \BX_{\dR}(\mathscr{D}))$.
As the image of $[c_{1,n}]$ in $H^1(G_K, \BB_{\log,E})$ is zero, so
is its image in $H^1(G_K, \BB_{\dR, E}/ \Fil^{f}\BB_{\dR,E})$, where
$f$ is the smallest integer such that $e^*_{1,z}\otimes e_{n,z}\in
\Fil^{-f}\mathscr{D}_K$. Hence, the image of $[\bar{c}']$ in
$H^1(G_K, \BX_{\dR}(\mathscr{D}))$ is zero.
\end{proof}

Now, applying Lemma \ref{lem:aux} to $\mathscr{D}_0$ with $f_1=
e^*_{n,z}\otimes e_{1,z}$, $f_2=e^*_{1,z}\otimes e_{1,z}$ and
$f_3=e^*_{n,z}\otimes e_{n,z}$, we get
$$ \gamma_{n,0}-\gamma_{1,0} = \sum_{\tau\in\mathrm{Emb}(K,E)} \CL_{\tau} (\gamma_{n,\tau}-\gamma_{1,\tau}).  $$
Hence, by Lemma \ref{lem:constant-gamma} we have
$$ \frac{v_p(\pi_K)}{[K_0:\BQ_p]} ( \epsilon_n(p) -\epsilon_1(p) ) + \sum_{\tau\in\mathrm{Emb}(K,E)} \CL_{\tau} (\epsilon_{n,\tau}-\epsilon_{1,\tau}) =0. $$
As $\frac{\mathrm{d} \delta_h (p) }{\delta_h(p)} = \epsilon_h(p)
\mathrm{d}Z$ and
$\mathrm{d}\vec{w}(\epsilon_h)=(\epsilon_{h,\tau}\mathrm{d}Z)_\tau$,
we obtain
$$ \frac{1}{[K:\BQ_p]}\left(\frac{\mathrm{d} \delta_n(p)}{\delta_n(p)} - \frac{\mathrm{d}
\delta_1(p)}{\delta_1(p)}\right)
 + \vec{\CL}_{\CF} \cdot (\mathrm{d}\vec{w}(\delta_n)-\mathrm{d}\vec{w}(\delta_1)) =
0,
$$ as desired. This finishes the proof of Theorem \ref{thm:main}.


\begin{thebibliography}{99}

\bibitem{Ben} D. Benois, {\it A generalization of Greenberg's
$\mathscr{L}$-invariant.} Amer. J. Math. 133 (2011), 1573-1632.

\bibitem{Ber02} L. Berger, {\it Repr\'esentations $p$-adiques et \'equations
diff\'erentielles}. Invent. Math. 148 (2002), 219-284.

\bibitem{Ber08} L. Berger, {\it Construction de $(\varphi,\Gamma)$-modules: repr\'esentations $p$-adiques et
$B$-paires}. Algebra Number Theory 2 (2008), 91-120.

\bibitem{Be2011} L. Berger, {\it  Trianguline representations.} Bull. Lond. Math.
Soc. 43 (2011),  no. 4, 619-635.

\bibitem{Cz2008} P. Colmez, {\it Repr\'esentations triangulines de dimension
$2$.} Asterisque 319 (2008), 213-258.

\bibitem{Cz2010} P. Colmez, {\it Invariants $\mathcal{L}$ et d\'eriv\'ees de valeurs propres de
Frobenius}. Ast\'erisque 331 (2010), 13-28.

\bibitem{Fontaine} J.-M. Fontaine, {\it Le corps des p\'eriodes $p$-adiques.} Ast\'erisque 223 (1994), 59-111.

\bibitem{GS} R. Greenberg, G. Steven, {\it $p$-adic L-functions and $p$-adic periods of modular
forms}. Invent. Math. 111 (1993), 407-447.

\bibitem{Mazur} B. Mazur, {\it The theme of
$p$-adic variation}, Mathematics: frontiers and perspectives, Amer.
Math. Soc., Providence 2000, 433-459.

\bibitem{MTT} B. Mazur, J. Tate, J. Teitelbaum, {\it On $p$-adic analogs of the conjectures of Birch and
Swinnerton-Dyer}. Invent. Math. 84 (1986), 1-48.

\bibitem{Naka} K. Nakamura, {\it Classification of two-dimensional split trianguline representations of $p$-adic
fields}. Compos. Math. 145 (2009), 865-914.

\bibitem{Pott} J. Pottharst, {\it The $\mathcal{L}$-invariant, the dual $\mathcal{L}$-invariant, and
families}.  Ann. Math. Qu\'e. 40 (2016), 159-165.

\bibitem{Xie2015} B. Xie, {\it Derivatives of Frobenius and Derivatives of
Hodge-Tate weights}. To appear in Acta Mathematica Sinica, a special
issue on Arithmetic Algebraic Geometry, edited by Fu, Liu, Tian and
Xu.

\bibitem{Zhang} Y. Zhang, {\it $\CL$-invariants and logarithmic derivatives of eigenvalues of
Frobenius.} Science China Mathematics 57 (2014), 1587-1604.


\end{thebibliography}
\end{document}